\documentclass[11pt]{amsart}
\usepackage{amsmath,amssymb,txfonts}
\usepackage{graphicx}
\usepackage{amscd,latexsym,amsthm,amsfonts,amssymb,amsmath,amsxtra}
\usepackage{amscd,mathrsfs,latexsym,amsmath,amsthm,amssymb,amsxtra,}
\usepackage{latexsym,amsmath,amsthm,amssymb,amsxtra,mathrsfs}
\usepackage[all]{xy}
\usepackage{amssymb}
\usepackage{amsxtra}
\usepackage{amsmath}
\usepackage{txfonts}
\newtheorem{prop}{Proposition}[section]
\newtheorem{theo}[prop]{Theorem}
\newtheorem{lem}[prop]{Lemma}

\newtheorem{cor}[prop]{Corollary}
\newtheorem{defi}[prop]{Definition}

\def\ZZ{{\mathbb Z}}

\def\RR{{\mathbb R}}
\def\CC{{\mathbb C}}

\begin{document}
\title{On the existence of conic K\"{a}hler-Einstein metrics}
\author{Gang $\text{Tian}^{*}$, Feng $\text{Wang}^{\dag}$}
\thanks { $*$ Partially supported by NSFC Grants 11331001.}
\thanks {  $\dag$  Partially supported by NSFC Grants 11501501 and
the Fundamental Research Funds for the Central Universities 2018QNA3001.}
\maketitle

\section{Introduction}
The Yau-Tian-Donaldson conjecture has been solved for Fano manifolds:
{\it If $M$ is a $K$-polystable Fano manifold, then it admits a K\"{a}hler-Einstein metric}.
Its first proof uses the continuity method through conic K\"ahler-Einstein metrics with conic angles along a smooth divisor. subsequently, there are proofs by Aubin's continuity method, the K\"ahler-Ricci flow as well as variational method.

Let us recall conic K\"ahler metrics along a smooth divisor.
First, we have the following model conic metric on $\mathbb{C}^n$:

$$\omega_{\beta}\,=\,\sqrt{-1}\,\left(\frac{dz_1 \wedge d\bar{z}_1}{|z_1|^{2-2\beta}}\,+\,\sum_{i= 2}^n\, dz_i\wedge d\bar{z}_i\right).$$

\begin{defi}
A conic K\"ahler metric $\omega$ with angle $2\pi\beta$ along a smooth divisor $D\subset M$ is simply a K\"ahler metric
$\omega$ on $M\backslash D$ such that for each $p\in D$, there is a local coordinate chart $U\subset \mathbb{C}^n$ around $p$ with local coordinates $z_1,\cdots,z_n$
such that
$\omega$ is equivalent to the model metric $\omega_\beta$ on $U$.
\end{defi}

In local coordinates $z_1,\cdots,z_n$, a conic metric $\omega$ is given by
$$\omega\,=\, \sqrt{-1}\,\sum_{i,j=1}^n \,g_{i\bar j}\,dz_i\wedge d\bar z_j,$$
then its Ricci curvature outside $D$ is given by
$$ {\rm{Ric}}(\omega)\,=\,-\,\sqrt{-1}\,\partial\bar\partial \log\det(g_{i\bar j}).$$
Using this formula, we can deduce that in the sense of currents,
$${\rm{Ric}}(\omega)\,=\,2\pi (1-\beta)[D]\,+\,\varphi, $$
where $\varphi$ is a ``nice'' (1,1)-form in certain sense, e.g., it has locally bounded potentials.

\begin{defi}
We say that $\omega$ is a conic K\"ahler-Einstein metric if
$${\rm{Ric}}(\omega)\,=\,\mu \omega\,+\,2\pi(1-\beta)\,[D],$$
in the sense of currents.
\end{defi}

If $M$ is a Fano manifold, then there exists a smooth divisor $ D\in |{-\lambda}K_M|$ for some large integer $\lambda>0$. Consider the following equation:
\begin{align}\label{conic1}
 {\rm{Ric}}(\omega)\,=\,\mu \omega\,+\,2\pi(1-\beta)\,[D],
\end{align}
where $\mu+(1-\beta)\lambda\,=\,1$. In the course of proving the Yau-Tian-Donaldson conjecture for Fano manifolds, one also showed the solvability of \eqref{conic1} for log-K-stable $(M, (1-\beta)D)$ (cf. \cite{Ti15}). The log K-stability was introduced in \cite{Li11} and extends the notion of K-stability introduced in \cite{Ti97}.

A natural question is to consider the existence of conic K\"{a}hler-Einstein metric on general log Fano manifolds. A
log Fano manifold is a tuple $(M, \sum_{i=1}^k (1-\beta^i) \,D_i)$, where $M$ is a K\"{a}hler manifold and $D_i$ are normal crossing smooth divisors such that $L$ is an ample $\mathbb{Q}$-line bundle, where $\beta_i\in(0,1)$ are rational numbers and
$$L\,=\,K_M^{-1}\,-\,\sum_{i=1}^k\,(1-\beta^i)\,D_i.$$
For simplicity, we often write $D$ for $\sum_{i=1}^k\,(1-\beta^i)\,D_i$ and $(M,D)$ for $(M, \sum_{i=1}^k (1-\beta^i) \,D_i)$.

A metric $\omega$ is called a conic K\"ahler metric if it is a smooth K\"ahler metric outside the support $|D|$ of $D$ and for each point $p\in |D|$ where $|D|$ is locally defined by the equation $z_1\cdots z_d = 0$ for some local coordinates $z_1,...,z_n$, the metric $\omega$ satisfies
$$ C^{-1}\,\omega_{cone}\, \leq\, \omega \,\leq \,C\,\omega_{cone},$$
where $C$ is a positive constant and $\omega_{cone}$ is the model conic metric with cone angles $2\pi\beta_i$ along $\{z_i = 0\}$, that is,
$$\omega_{cone}\,=\,\sum_{i=1}^d \,\sqrt{-1}\,\frac{dz_i\wedge d \bar z_i}{|z_i|^{2(1-\beta_i)}}\,+\,\sum_{k=d+1}^n \,\sqrt{-1}\,dz_i\wedge d \bar z_i.$$

A conic K\"ahler metric is called a conic K\"{a}hler-Einstein metric if it satisfies:
$${\rm Ric}(\omega)\,=\,\omega\,+\,2\pi\,D,$$
where $D= \sum_{i=1}^k\,(1-\beta^i)\,D_i$. It is proved in \cite{Ber} that if $(M,D)$ admits a conic K\"{a}hler-Einstein metric, then it is log-$K$-polystable as defined in \cite{Li11}. So we want to prove the converse, that is, the log $K$-polystability is also sufficient for the existence of conic K\"{a}ler-Einstein metrics with conic angle $2\pi \beta_i$ along each $D_i$.

One of our motivations comes from the study of the existence of K\"{a}hler-Einstein metrics on $K$-polystable singular Fano variety.
Such an existence problem can be reduced to studying the existence and compactness of general conic K\"{a}ler-Einstein metrics.
Let $X$ be a $\mathbb{Q}$-Fano variety. Assume that there exists a log resolution $\mu: M\rightarrow X$ such that
\[
K_M\,=\,\mu^* K_X\,+\,\sum_i\, a_i E_i
\]
with $a_i\in (-1, 0]$. So we have
\[
K_M^{-1}\,=\,\mu^*K_X^{-1}\,+\,\sum_i\, b_i E_i
\]
with $b_i\in [0, 1)$. Then
\[
L_\epsilon\,=\,\mu^*K_X^{-1}\,-\,\sum_i \epsilon_i E_i\,=\,K_M^{-1}\,-\,\sum_i\, (b_i\,+\,\epsilon_i)\, E_i
\]
is an ample $\mathbb{Q}$-line bundle for suitable small rational numbers $\epsilon_i$. If there is a K\"{a}hler-Einstein metric $\omega$ on $X$, then $\mu^*\omega$ will be a degenerate conic K\"{a}hler-Einstein metric with cone angle $2\pi(1-b_i)$ along $E_i$ in its normal direction:
$${\rm Ric}( \mu^*\omega)\,=\,\mu^*\omega\,+\,2\pi\sum_i\, b_i\,E_i.$$
So for a $K$-stable $\mathbb{Q}$-Fano variety with such a resolution, we first expect that $\mu^*\omega$ can be perturbed to obtain a sequence of
conic K\"{a}hler Einstein metrics on $M$:
$${\rm Ric}( \omega_\epsilon)\,=\,\omega_\epsilon\,+\,2\pi\sum_i \,(b_i+\epsilon)\,E_i,~~~[\omega_\epsilon]\,=\,2\pi \,c_1(L_\epsilon).$$
Then we expect that the Gromov-Hausdorff limit of $(M,\omega_\epsilon)$ coincides with the K\"{a}hler-Einstein metric on $X$. This is indeed true and we can turn around to use this to construct a K\"ahler-Einstein metric on $X$.

Our main theorem of this paper is the following
\begin{theo}\label{main}
If $(M,\Sigma_{i=1}^k(1-\beta^i)D_i)$ is log K-polystable, then there exists a conic K\"{a}hler-Einstein metric with angle $2\pi \beta^i$ along $D_i$.
\end{theo}

We will follow the approach in \cite{Ti15} to prove Theorem \ref{main}. At first, we approximate the conic K\"{a}hler-Einstein metric by smooth metrics with Ricci curvature bounded from below. If each $D_i$ is semi-ample, then this approximation has been done in \cite{S} by extending the method introduced in \cite{Ti15} which uses the semi-ampleness of $D_i$ in a crucial way. However, in our case, the conic divisors $D_i$ may not be semi-ample, so we need to use a variational method to do the approximation. With the approximation, we can apply Cheeger-Colding's theory and Cheeger-Colding-Tian's theory. Secondly, adapting the arguments in \cite{Ti15} to our case here, we establish the partial $C^0$ estimate. In our new situation here, not all the angles approach $2\pi$ as one had in the smooth case, this causes some problems of extending certain arguments from the smooth case to our more general cases. At last, we prove the existence of conic K\"{a}hler-Einstein metric by using log $K$-polystability.

\section{Smoothing conic K\"{a}hler-Einstein metric}

In this section, we prove an approximation theorem for general conic K\"ahler metrics. Our theorem extends the main result in \cite{S}
under the assumption of semi-ampleness of the conic divisors.

Let $(M,D)$ be a log Fano manifold and $\omega$ be a conic K\"{a}hler-Einstein metric, where
$D\,=\,\sum_{i=1}^k\,(1-\beta^i)\,D_i$ and $\beta^i\in (0,1)$.

We say that $\omega$ has a $K$-approximation if there is a sequence of smooth metrics $\omega_i=\omega+\sqrt{-1}\partial\bar{\partial}\,\phi_i$
satisfying:

\vskip 0.1in
\noindent
i) $\phi_i$ converges to $0$ uniformly in $M$ and smoothly outside the support $|D|$ of $D$;

\vskip 0.1in
\noindent
ii) ${\rm Ric}( \omega_i)\,\geq\, K\,\omega_i$;

\vskip 0.1in
\noindent
iii) $(M,\omega_i)$ converges to $(M,\omega)$ in the Gromov-Hausdorff topology.
\vskip 0.1in

\begin{theo}\label{ricci}
Assume that $Aut^0(M,D)=1$ and $(1-K_i)\,L+(1-\beta^i)\,D_i$ are all semi-positive for some $K_i\leq 1$ ($1\leq i\leq k$),
then $\omega$ has a $K$-approximation,
where $K=\Sigma (K_i-1)+1$.
\end{theo}
It is easy to see that the condition in Theorem \ref{ricci} is necessary when $k=1$. Now we fix a smooth metric $\omega_0\in 2\pi\, c_1(L)$.
For each $i$, since
$(1-K_i)L+(1-\beta^i)D_i$ is semi-positive, there is a smooth Hermitian metric $h_{i}$ on $[D_i]$ such that
$$(1-K_i)\,\omega_0\,+\,(1-\beta^i)\,R(h_i)\,\geq \,0,$$
where $R(h_i)$ is the curvature form of $h_{i}$.
Choose a smooth volume form $\Omega$ such that
$${\rm Ric}\, (\Omega)\,=\,\omega_0\,+\,\sum_{i=1}^k\,(1-\beta^i)\,R({h_{i}}),~~~~ \int _M \,\Omega\,=\,\int_M\,\omega_0^n\,=\,V.$$
Then the conic K\"ahler-Einstein metric $\omega$ can be written as $\omega_0+\sqrt{-1}\,\partial\bar{\partial}\,\phi_{KE}$ for some function $\phi_{KE}$ satisfying:
\begin{align}\label{conic}
(\omega_0\,+\,\sqrt{-1}\,\partial\bar{\partial}\,\phi_{KE})^n\,
=\,e^{-\phi_{KE}}\,\left(\prod_{i=1}^k\,|S_i|_{h_{i}}^{-2(1-\beta^i)}\right)\,\Omega,
\end{align}
where $S_i$ is a defining section of $D_i$. Consider the following perturbed equation:
\begin{align}\label{app}
(\omega_0\,+\,\sqrt{-1}\,\partial\bar{\partial}\,\phi_\delta)^n\,=\,\lambda\, e^{-\phi_\delta}\,\left(\prod_{i=1}^k\,\left(|S_i|_{h_i}^2\,+\,\delta \, e^{-\frac{1-K_i}{1-\beta^i}\,\phi_\delta}\right)^{-(1-\beta^i)} \right)\,\Omega.
\end{align}
where $\lambda$ is some positive constant which may depend on $\delta$.
\begin{prop}
There are constants $a,b,\delta_0 > 0$ depending only on $(M,\omega_0), \Omega, S_i, h_{i})$ such that for any $\delta\leq\delta_0$,
(\ref{app}) has a solution
$\omega_\delta$ with some $\lambda\in [a,b]$.
\end{prop}

We will use the variational method to prove this proposition. Following \cite{BEGZ}, we define
$$PSH_{full}\,=\,\{\phi\in PSH(M,\omega_0)\,|\,\lim_{j\rightarrow \infty}\int_{\phi\,\le\, -j}\,(\omega_0\,+\,\sqrt{-1}\,\partial\bar{\partial}\,\max\{\phi,-j\})^n\,=\,0\}$$
and
$$\mathcal{E}^1(M,\omega_0)\,=\,\{\phi\in PSH_{full}(M,\omega_0)\,|\,\phi\,\in\, L^1(\omega_\phi)\}.$$
The topology on these spaces are all weak topology, i.e. the $L^1$-topology.

It follows from Lemma 6.4 in \cite{BBGZ}
\begin{lem}\label{lemm:bbgz-1}
The map $$\mathcal{E}^1(M,\omega_0)\rightarrow L^1(M,\omega_0):\phi\rightarrow e^{-\phi}$$
is continuous.
\end{lem}
Now we define
$$\mathcal{H}\,=\,\{\phi\,\in \,\mathcal{E}^1(M,\omega_0)\,|\,\int_M \,h(e^{-\phi})\,\Omega\,=\,\int_M \,h(1)\,\Omega\},$$
where
$$h(x)\,=\,\int_0^x\,\prod_{i=1}^k\,\left(|S_i|_{h_{i}}^2\,+\,\delta\, t^{\frac{1-K_i}{1-\beta^i}}\right)^{-(1-\beta^i)}\,dt.$$
By Lemma \ref{lemm:bbgz-1}, $\mathcal{H}$ is a closed subset of $\mathcal{E}^1(M,\omega_0)$.

We have two functionals on $\mathcal{H}$:
$$J(\phi)\,=\,\frac{1}{V}\,\int_M\,\phi\,\omega_0^n\,-\,\frac{1}{(n+1)V}\,\sum_{i=0}^n\,\int_M\,\phi\,\omega_0^i\wedge\omega_\phi^{n-i},$$
$$F_\delta(\phi)\,=\,J(\phi)\,-\,\frac{1}{V}\,\int_M\,\phi\,\omega_0^n\,-\,\log \left(\int_M \,h(e^{-\phi})\,\Omega\right).$$
It is easy to see that
$$F_\delta(\phi)\,=\,J(\phi)\,-\,\frac{1}{V}\,\int_M\,\phi\,\omega_0^n\,+\,F_\delta(0),~~~~ F_\delta(0)\,=\,-\log \int_M \,h(1)\,\Omega.$$
For $\delta\leq 1$, $F_\delta(0)$ is uniformly bounded by a constant depending only on $(M,\omega_0, \Omega, S_i, h_i)$.
\begin{lem}\label{lsc}
$J(\phi)$ is lower semi-continuous on $\mathcal{H}$.
\end{lem}
\begin{proof}
By Proposition 2.10 in \cite{BEGZ}, we know that $J(\phi)$ is lower semi-continuous on $\mathcal{E}^1(M,\omega_0)$. Since $\mathcal{H}$ is a closed subset of $\mathcal{E}^1(M,\omega_0)$, the lemma is proved.
\end{proof}
Now we prove the proposition.
\begin{proof}
From Theorem 2.12 in \cite{DR}, we know that log Ding functional
$$F_0(\phi)\,=\,J(\phi)\,-\,\frac{1}{V}\,\int_M\,\phi\,\omega_0^n\,-\,\log\left(\int_M \,\prod_{i=1}^k\,|S_i|_{h_{i}}^{-2(1-\beta^i)}\,e^{-\phi}\,\Omega\right)$$
is coercive, that is, for some positive constants $A$ and $B$, we have
$$F_0(\phi)\,\geq \,A\,J(\phi)\,-\,B.$$
Clearly, $F_\delta\,\geq\, F_0$, so $F_\delta$ is also coercive and consequently, there is a minimizing sequence $\phi_j$ of $F_\delta$ satisfying:
$$\lim_{j\to \infty}\,F_\delta(\phi_j)\,=\,\inf_{\phi \in {\mathcal H}} F_\delta(\phi).$$
By the above coercive inequality, for $j$ sufficiently large, we have
\begin{align}
 J(\phi_j)\,\leq \,\frac{1}{A}\,(F_\delta(\phi_j)\,+\,B)\,\leq \,\frac{1}{A}\,(F_\delta(0)\,+\,B) + 1,\label{J}
  \end{align}
Hence,
\begin{align}
\left|\,\int_M\,\phi_j\,\omega_0^n\,\right |\,\le\,|J(\phi_j)|\,+\,|F_\delta(\phi_j)|\,+\,|F_\delta(0)|\,\leq\, C(A,B,F_\delta(0))\label{int}
 \end{align}
So we have
$$|\sup \phi_j|\,\leq\, C(A,B,F_\delta(0)).$$
From this and (\ref{J}), we know that $\phi_j$ lies in a weakly compact subset $\mathcal{K}\subset \mathcal{E}^1(M,\omega_0)$. Note that $\mathcal{K}$ is independent of $\delta$. Hence,
by taking a subsequence if necessary, we may assume that $\phi_j$ converge to a limit $\phi_\delta$ in $\mathcal{E}^1(M,\omega_0)$. From Lemma \ref{lsc}, we know that $\phi_\delta$ is a minimizer of $F_\delta$. By using the arguments in the proof of Theorem 4.1 in [BBGZ], we can further show that $\phi_\delta$ is a critical point of $F_\delta$, so it satisfies (\ref{app}) for some $\lambda$.

Next, we estimate $\lambda$. From (\ref{int}), we know that
 \begin{align}
\int_M\,|\phi_j|\,\omega_0^n\,\leq\, C(A,B,F_\delta(0),V).
 \end{align}
It follows that
$$\left |\{e^{-\phi_j}\,\geq\, C_1\}\right|\,=\,\left |\{\phi_j\,\leq\,-\, \ln C_1\}\right |\,\leq\,\frac{\left |\int_M\,\phi_j\,\omega_0^n\right |}{\ln C_1}\,\leq\, \frac{C(A,B,F_\delta(0),V)}{\ln C_1}.$$
So we can choose $C_1\,>\,0$ such that
$$\left |\{e^{-\phi_j}\,\geq\, C_1\}\right |\,\leq\,\frac{V}{4}.$$
We also choose $\epsilon>0$ such that
$$\left |\{|S_{i}\right |_{h_{i}}^2\,\leq\, \epsilon\}|\,\leq\, \frac{V}{4k}, ~~~~{\rm where}~~i\,=\,1,\cdots,k.$$
Put
$$N\,=\,\{e^{-\phi_j}\,\leq \,C_1\}\,\cap\,\left(\cap_{i=1}^k \, \{|S_{i}|_{h_{i}}^2\,\geq\, \epsilon\}\right),$$
then
$$|N|\,\geq\,\frac{V}{2}.$$
On $N$, there is a $\delta_0\,=\,\delta_0(M,\omega_0, \Omega, \{S_i\}, \{h_i\})$ such that for any $\delta \le \delta_0$ and each $i=1,\cdots,k$, we have
$$e^{-C(A,B,F_\delta(0))}\,\leq\, e^{-\phi_j}\,\leq\, C_1$$
and
$$(|S_{i}|_{h_{i}}^2\,+\,\delta\, e^{\frac{1-K_i}{1-\beta^i}\,\phi_j})^{-(1-\beta^i)}\,\geq\, \frac{1}{2}\, |S_{i}|_{h_{i}}^{-2(1-\beta^i)}.$$
So we get
$$\int_N \,e^{-\phi_\delta}\,\prod_{i=1}^k\left(|S_{i}|_{h_{i}}^2\,+\,\delta \,e^{-\frac{1-K_i}{1-\beta^i}\,\phi_\delta}\right)^{-(1-\beta^i)}\,\Omega\,\geq \,C_2(M,\omega_0, \Omega,\{ S_i\}, \{h_i\}),$$
and consequently
$$\lambda\,\leq\, \frac{V}{C_2(M,\omega_0, \Omega,\{S_i\}, \{h_i\})}.$$
On the other hand, we have
$$e^{-\phi_\delta}\,\prod_{i=1}^k\,\left (|S_i|_{h_{i}}^2\,+\,\delta\, e^{-\frac{1-K_i}{1-\beta^i}\,\phi_\delta}\right )^{-(1-\beta^i)}\,\leq\, h(e^{-\phi_\delta}),$$
we have
$$\int_M \,e^{-\phi_\delta}\,\prod_{i=1}^k\,\left(|S_i|_{h_{i}}^2\,+\,\delta\, e^{-\frac{1-K_i}{1-\beta^i}\,\phi_\delta}\right)^{-(1-\beta^i)}\,\Omega\,\leq\, \int_M\,h(e^{-\phi_\delta})\,\Omega\,=\,\int_M \,h(1)\,\Omega.$$
Hence,
$$\lambda\,\geq\, \frac{1}{V}\,\int_M \,h(1)\,\Omega.$$
This completes the proof of our proposition.
\end{proof}
Before we show the regularity of the solution $\phi_\delta$, let us first calculate the Ricci curvature of $\omega_\delta$.
\begin{lem}
Assuming that $\omega_\delta$ is a smooth metric, then the Ricci curvature of $\omega_\delta$ satisfies
$${\rm Ric} ( \omega_\delta)\,\geq\, K\,\omega_\delta.$$
\end{lem}
\begin{proof}
Write (\ref{app}) as:
$$(\omega_0\,+\,\sqrt{-1}\,\partial\bar{\partial}\,\phi_\delta)^n\,=\,\lambda\, e^{(\sum (1-K_i)-1)\,\phi_\delta}\,\prod_{i=1}^k\,\left(|S_i|_{h_{i}}^2\,e^{\frac{1-K_i}{1-\beta^i}\,\phi_\delta}\,+\,\delta \right)^{-(1-\beta_i)}\,\Omega.$$
Set
$$h^{\delta}_{i}\,=\,h_{i}\, e^{\frac{1-K_i}{1-\beta^i}\,\phi_\delta}.$$
Clearly, we have
$$R(h^{\delta}_{i})\,=\,R(h_{i})\,-\,\frac{1-K_i}{1-\beta^i}\,\sqrt{-1}\,\partial\bar{\partial}\,\phi_\delta.$$
Then the Ricci curvature ${\rm Ric}(\omega_\delta)$ is equal to
\begin{align}
&,\sqrt{-1}\,\left((1\,+\,\sum_{i=1}^k\,(K_i-1))\,\partial\bar{\partial}\,\phi_\delta\,+\,
\sum_{i=1}^k\,(1-\beta^i)\,\partial\bar{\partial}\,
\log(|S_i|_{h^{\delta}_{i}}^2\,+\,\delta)\right)\,+\,{\rm Ric}(\Omega) \notag\\
\geq &\, (1\,+\,\sum_{i=1}^k\,(K_i-1))\,\sqrt{-1}\,\partial\bar{\partial}\,\phi_\delta
\,-\,\sum_{i=1}^k\,(1-\beta^i)\,\frac{|S_i|_{h^{\delta}_{i}}^2\,R(h^{\delta}_{i})}{|S_i|_{h^{\delta}_{i}}^2
\,+\,\delta}\,+\,{\rm Ric}(\Omega) \notag\\
=&\, (1\,+\,\sum_{i=1}^k\,(K_i-1))\,\sqrt{-1}\,\partial\bar{\partial}\,\phi_\delta
\,-\,\sum_{i=1}^k\,(1-\beta^i)\,\frac{|S_i|_{h^{\delta}_{i}}^2\,R(h_{i})}{|S_i|_{h^{\delta}_{i}}^2\,+\,\delta}\notag\\
&+\Sigma_{i=1}^k(1-K_i)\,\frac{|s_{D_i}|_{h^{\phi_\delta}_{D_i}}^2\,\sqrt{-1}\partial\bar{\partial}\phi_\delta}{|s_{D_i}|_{h^{\phi_\delta}_{D_i}}^2+\delta}\notag
\,+\,\omega_0\,+\,\sum_{i=1}^k\,(1-\beta^i)\,R(h_{i})\notag\\
=&\sum_{i=1}^k\,\frac{\delta\,[(1-\beta^i)\,R(h_{i})\,+\,(1-K_i)\,\omega_0]}{|S_i|_{h^{\delta}_{i}}^2\,+\,\delta}\,+\,\sum_{i=1}^k\,(K_i-1)
\,\frac{\delta}{|S_i|_{h^{\delta}_{i}}^2\,+\,\delta}\,\omega_\delta
\,+\,\omega_\delta \notag\\
\geq &\, \left (\sum_{i=1}^k\,(K_i-1)\,+\,1\right )\,\omega_\delta \notag.
\end{align}
This lemma is proved.
\end{proof}
Next we establish the regularity of the solution $\phi_\delta$.
\begin{lem}
For some $\alpha=\alpha(M,\omega) \in (0,1)$, we have
$$|\phi_\delta|_{C^\alpha(M,\omega_0)}\,\leq\, C_3,$$
where $C_3= C_3(M,\omega_0, \Omega, \{S_i\}, h_0)$.
\end{lem}
\begin{proof}
From the above, we know that
$$\phi_\delta\,\in\,\mathcal{K}\,\subset\,PSH_{full}(M,\omega_0),$$
where $\mathcal{K}$ is a weak compact subset. So by [Zer], for any $p>1$, $|e^{-\phi_\delta}|_{L_p}$ is uniformly bounded by some constant $C_p$ depending on $\mathcal{K}$. For any $q\in (1,\min\{\frac{1}{1-\beta^i}\})$,
we have
$$\left |\prod_{i=1}^k\,\left (|S_i|_{h_{i}}^2\,+\,\delta\, e^{\frac{1-K_i}{1-\beta^i}\,\phi_\delta}\right)^{-(1-\beta^i)}\right |_{L^q}\,\leq\,
\left |\prod_{i=1}^k\,|S_i|_{h_{i}}^{-2(1-\beta^i)}\right |_{L^q}\,\leq\, C(M,\omega_0, \Omega,\{S_i\}, \{h_{i}\},q).$$
It follows from the H\"{o}lder inequality that for any $p\in (1,\min\{\frac{1}{1-\beta^i}\})$ and some constant $C$ independent of $\delta$,
$$\left |e^{-\phi_\delta}\,\prod_{i=1}^k\,\left (|S_i|_{h_{i}}^2\,+\,\delta \,e^{\frac{1-K_i}{1-\beta^i}\,\phi_\delta}\right)^{-(1-\beta^i)}\right |_{L_p}\,\leq\, C.$$
So the lemma follows.
\end{proof}
\begin{prop}\label{cc}
There exists $\delta_l \rightarrow 0$ such that $\phi_{\delta_l}$ converges to $\phi_{KE}+c$ in the $C^0$-topology for some constant $c$.
\end{prop}
\begin{proof}
From the above lemma, we can choose a subsequence $\phi_{\delta_l}$ which converges to a continuous function $\psi$, moreover, for some $\lambda$,
$\psi$ satisfies
$$(\omega_0\,+\,\sqrt{-1}\,\partial\bar{\partial}\,\psi)^n\,=\,\lambda \,e^{-\psi}\,\prod_{i=1}^k \,|S_i|_{h_{i}}^{-2(1-\beta^i)}\,\Omega.$$
By the uniqueness result of Berndtsson in \cite{Bo}, we know that $\psi\,=\,\phi_{KE}+c$.
\end{proof}

Now we show that $\phi_\delta$ is a smooth function. We need a special case of Proposition 2.1 in \cite{GP}.
\begin{prop}\label{regular}
Let $\phi$ be a solution of
$$\omega_\phi^n\,=\,e^{\psi^{+}-\psi^{-}}\,\omega_0^n,$$
where $\omega_\phi = \omega_0+\sqrt{-1}\,\partial\bar{\partial}\,\phi$ and $\psi^{\pm}$ are smooth functions. We further
assume that there exists a $C>0$ such that:

\vskip 0.1in
\noindent
(i) $ |\phi|\, \leq\, C$;

\vskip 0.1in
\noindent
(ii) $|\psi^{\pm}|\,\leq\, C $ and $\sqrt{-1}\,\partial\bar{\partial}\,\psi^{\pm}\,\geq\, -\,C\,\omega_0$;

\vskip 0.1in
\noindent
(iii) The curvature tensor $R(\omega_0)$ of $\omega_0$ is bounded from below by $-C$.

\vskip 0.1in
\noindent
Then there exists a constant $A>0$ depending only on $C$ such that
$$\frac{1}{A}\,\omega_0\,\leq\, \omega_\phi\,\leq\, A\,\omega_0.$$
\end{prop}
Choose a sequence of smooth $\omega_0$-psh functions $\psi_j$ which converges to $\phi_\delta$ in the $C^0$-norm. Note that by saying that $\psi_j$ is $\omega_0$-psh, we mean
$$\omega_0+\sqrt{-1}\,\partial\bar\partial \,\psi_j \,\ge\, 0.$$
\begin{lem}
If $\phi_j$ is any solution of
\begin{align}\label{smooth}
(\omega_0\,+\,\sqrt{-1}\,\partial\bar{\partial}\,\phi_j)^n\,=\,\lambda e^{-K\,\psi_j}\,\prod_{i=1}^k\, \left(|S_i|_{h_{i}}^2\,e^{\frac{1-K_i}{1-\beta^i}\,\psi_j}\, +\,\delta\right )^{-(1-\beta^i)}\,\Omega,
\end{align}
then for some $C\,=\,C(M,\delta,|\phi_\delta|_{C^0})$,
$$|\Delta \phi_j|\,\leq\, C.$$
\end{lem}
\begin{proof}
Put
$$u_j\,=\,\sum_{i=1}^k \,(1-\beta^i)\,\log\left (|S_i|_{h_{i}}^2\,e^{\frac{1-K_i}{1-\beta^i}\,\psi_j}\, +\,\delta\right).$$
For each $i$, we put for simplicity
$$||S_i||_j^2\,=\,|S_i|_{h_i}^2\,e^{\frac{1-K_i}{1-\beta^i}\,\psi_j}.$$
Using the computation as above, we know that,
\begin{align}
\sqrt{-1}\,\partial\bar{\partial}\, u_j &\geq\, -\sum_{i=1}^k\,(1-\beta^i)\frac{||S_i||_j^2}{||S_i||_{j}^2\,+\,\delta}\,R(h_{i})\,+\,
\sum_{i=1}^k\,(1-K_i)\frac{||S_i||_{j}^2}{||S_{i}||_{j}^2\,+\,\delta}\,\sqrt{-1}\,\partial\bar{\partial}\,\psi_j\notag\\
&\geq\, \sum_{i=1}^k\,\frac{\delta}{||S_i||_{j}^2\,+\,\delta}\,\left((1-\beta^i)\,R(h_i)\,+\,(1\,-\,K_i)\,\omega_0\right)\notag\\
&-\,\sum_{i=1}^k\, (1-\beta^i)\,R(h_i) \,-\,(1\,-\,K) \,\omega_0\notag
\end{align}
Here we have used the fact:
$$\omega_0\,+\,\sqrt{-1}\,\partial\bar{\partial}\,\psi_j \,\geq\, 0.$$
By our assumption,
$$(1-\beta^i)\,R(h_i)\,+\,(1\,-\,K_i)\,\omega_0\,\geq\, 0,$$
so for some $C>0$, we have
$$\sqrt{-1}\,\partial\bar{\partial}\, u_j\,\geq\, -\,C\,\omega_0.$$
Since $K = \sum_{i=1}^k (K_i-1) + 1 \leq 1$, we have
$$\sqrt{-1}\,\partial\bar{\partial}\,((-K+1)\,\psi_j)\, \geq\, (K-1)\,\omega_0.$$
The righthand side of $(\ref{smooth})$ can be written as
$e^{\psi_{+,j}-\psi_{-,j}}$, where
$$\psi_{+,j}\, =\, (1-K)\, \psi_j,\,\,\,\,{\rm and}\,\,\,\,\,\psi_{-,j}\,=\,u_j\,+\,\psi_j.$$
It follows from the above that for some constant $C > 0$, we have
$$\sqrt{-1}\,\partial\bar{\partial}\,\psi_{\pm,j}\geq - C \omega_0.$$
Hence, by Proposition \ref{regular}, we have
$$|\Delta \phi_j|\leq C_\delta,$$
where $C_\delta = C(M,\delta,|\phi_\delta|_{C^0})$.
\end{proof}
It follows from the uniqueness theorem for complex Monge-Ampere equations that $\phi_i$ converges to $\phi_\delta + c$ for some constant $c$, so we have  $$|\phi_\delta|_{C^{1,1}(M,\omega_0)}\,\leq\, C_\delta.$$
By the Evans-Krylov theory, we know that for some $\alpha \in (0,1)$,
$$|\phi_\delta|_{C^{2,\alpha}(M,\omega_0)}\,\leq\, C'_\delta,$$
where $C'_\delta$ may depend on $\delta$. Then, using bootstrap, we see that $\phi_\delta$ is a smooth function.

Now we derive uniform estimates of $\omega_\delta$ and prove the Gromov-Hausdorff convergence.

\begin{lem}
There exists $C = C(M,\omega_0,|\phi_\delta|_{C^0})$ such that
$$\frac{1}{C}\,\omega_0\,\leq \,\omega_\delta\,\leq\, C\,\prod_{i=1}^k\,|S_i|_{h_{i}}^{-2(1-\beta^i)}\,\omega_0.$$
\end{lem}
\begin{proof}
Since the Ricci curvature ${\rm Ric}(\omega_\delta)$ is bounded from below by $ K\,\omega_\delta$, by the Chern-Lu inequality, we have
$$\Delta_{\omega_\delta} \log {\rm tr}_{\omega_\delta} \omega_0\,\geq\, -\,K\,-\,B\,{\rm tr}_{\omega_\delta}\omega_0,$$
where $B$ is the upper bound of the bisectional curvature of $\omega_0$. Then we have
$$\Delta_{\omega_\delta}(\log {\rm tr}_{\omega_\delta}\omega_0\,-\,(B+1)\,\phi_\delta)\,\geq\, {\rm tr}_{\omega_\delta}\omega_0\,-\,n\,(B-1)\,-\,K.$$
So by the Maximum Principle, we get the lower bound of $\omega_\delta$. Then the upper bound follows from (\ref{app}) and lower bound of $\omega_\delta$.
\end{proof}
By the regularity theory on higher derivatives, $\phi_\delta$ converges to $\phi$ smoothly outside $|D|$.
The following proposition is standard to experts in the field in view of what we had above. It can be proved by following the arguments in the proof of a similar result in \cite{Ti15}.
\begin{prop}\label{ghc}
For the sequence $\phi_{\delta_i}$ as in Proposition \ref{cc}, $(M,\omega_{\delta_i})$ converges to $(M,\omega)$ in the Gromov-Hausdorff sense.
\end{prop}
\begin{proof}
Let us outline a proof here for the readers' convenience. Fix a ball $B_p(r,\omega) \subset M\setminus D$. Since $\omega_{\delta_i}$ converges smoothly outside $D$, $B_p(r,\omega_{\delta_i})$ converges to $B_p(r,\omega)$ smoothly. In particular, the volume of $B_p(r,\omega_{\delta_i})$ has a definite lower bound. So $(M,\omega_{\delta_i})$ is locally non-collapsing, and we can consider the pointed Gromov-Hausdorff convergence. Taking any convergent subsequence of $(M,\omega_{\delta_i})$, denote the limit by $(X,d)$. Because $M\setminus \bigcup_{i=1}^k D_i$ is a convex dense subset of $M$ , and $\omega_{\delta_i}$ converges to $\omega$ smoothly in this set, we will have a 1-Lipschitz map
$$\iota: (M,\omega)\rightarrow (X,d),$$ which is a local isometry on $M\setminus D$. By the volume convergence, $\iota(M\setminus D)$ is a full $\mathcal{H}^n$ measure subset of $X$. Then we know that $\iota(M\setminus D)$ is a open and dense subset in the regular part of $X$. By Theorem 3.7 in \cite{CC2}, $\iota(M\setminus D)$ is weakly convex, i.e. for any point $x\in \iota(M\setminus D)$, there is a subset $C_x \subset X$ of measure zero such that there is a minimal geodesic from $x$ to any point outside $C_x$. It follows that the induced distance on $\iota(M\setminus D)$ is equal to the intrinsic distance. So $\iota$ is an isometry.
\end{proof}
\section{Basic estimates for conic K\"{a}hler-Einstein metrics}
Suppose that $\omega$ is a conic K\"{a}hler-Einstein metric on $M$ satisfying:
$${\rm Ric}(\omega)\,=\,t\,\omega\,+\,\sum_{i=1}^k\,(1-\beta^i)\,D_i,~~~ \omega\,\in\, 2\pi\, c_1(L),$$
where $L$ is $\mathbb{Q}$-line bundle and $t\in [\mu,+\infty)$ for some $\mu>0$. We are going to prove some basic estimates which will be used later. Denote the volume of $\omega$, which is a cohomological constant,  by
$$ V\,=\,(2\pi)^n\,\int_M\,c_1(L)^n.$$
Since $C^\infty_0(M\setminus \cup_{i=1}^k D_i)$ is dense in $H^1(M,\omega)$ and the Ricci curvature of $\omega$ is bounded from below by a positive constant, we
have that the Sobolev constant $C_s$ of $\omega$ is uniformly bounded by a constant $C=C(n,\mu,V).$

Now we give an application of the Moser iteration to conic K\"{a}hler-Einstein metrics.
\begin{lem}\label{moser}
Assume that $f$ is a non-negative function satisfying:
$$f\,\in \,L^\infty(M)\cap C^\infty(M\setminus \cup_{i=1}^k D_i)$$
and for some positive constant $A\geq1$,
\begin{equation}\label{inequ-0}
\Delta f\,\geq\, - \,A\,f,~~ \text{ in }M\setminus \cup_{i=1}^k D_i.
\end{equation}
Then we have
\begin{align}\label{inequ}
|f|_{L^\infty(M)}\,\leq\, C \,A^n\,|f|_{L^1(M,\omega)},
\end{align}
where $C$ is a constant depending on dimension $n$ and the Sobolev constant $C_S$ defined by: For any $f\in H^1(M,\omega)$, we have
$$\left(\int_M \,|f|^{\frac{2n}{n-1}}\,\omega^n\right)^{\frac{n-1}{n}}\,\leq\, C_S\,\left(\int_M\, |\nabla f|^2\,\omega^n\,+\,\int_M \,|f|^2\,\omega^n\right). $$
\end{lem}
\begin{proof}
As before, we denote by $h_i$ a smooth Hermitian metric on the line bundle $L_{D_i}$ associated to $D_i$ and $S_i$ a defining section of $D_i$.
Let $\tilde\eta$ be a cut-off function on $\mathbb R$ satisfying:
$\tilde\eta(t)\,=\,1$ for $t\geq 2$, $\tilde\eta(t)\,=\,0$ for $t\leq 1$,
and $|\tilde\eta'|\,\le\, 2$.

Define
$$\gamma_\epsilon\,=\,\prod_{i=1}^k\,\tilde\eta(\epsilon \log(-\log |S_i|_{h_i}^2)).$$
Then $\gamma_\epsilon$ has its support in $M\setminus \cup_{i=1}^k D_i$ satisfying:
$$\int_M\,|\nabla \gamma_\epsilon|^2\,\omega^n\leq C\epsilon,$$
where $C$ is a constant. Multiplying both sides of (\ref{inequ-0}) by $\gamma_\epsilon^2 f^{p-1}$, we have
$$-\,\gamma_\epsilon^2\, f^{p-1}\,\Delta f\, \leq \,A\,\gamma_\epsilon^2\, f^p.$$
Integrating by parts, we have
\begin{align}\label{integral}
\frac{4(p-1)}{p^2}\int_M \,\gamma_\epsilon^2\, |\nabla f^{\frac{p}{2}}|^2\,\omega^n\,+\,2\,\int_M\, \gamma_\epsilon f^{p-1} \langle \nabla f,\nabla \gamma_\epsilon\rangle\,\omega^n\,\leq\, A\,\int_M \,\gamma_\epsilon^2 f^p\,\omega^n.
\end{align}
Assume that $|f|\leq Q$, we have
\begin{align}
\left|\int_M \,\gamma_\epsilon \,f^{p-1}\,\langle \nabla f,\nabla \gamma_\epsilon\rangle\,\omega^n\right |&\leq\, \left(\int_M\, \gamma^2_\epsilon \,f^{2p-2}\,|\nabla f|^2\,\omega^n\right)^\frac{1}{2}\left(\int_M\,|\nabla \gamma_\epsilon|^2\,\omega^n\right)^\frac{1}{2} \notag\\
&\leq\, Q^\frac{p}{2}\,\left(\int_M \,\gamma^2_\epsilon \,f^{p-2}\,|\nabla f|^2\,\omega^n\right)^\frac{1}{2}\left(\int_M\,|\nabla \gamma_\epsilon|^2\,\omega^n\right)^\frac{1}{2}\notag\\
&\leq\, C\,Q^\frac{p}{2}\,\sqrt{\epsilon}\,\left(\int_M \,\gamma^2_\epsilon \,f^{p-2}\,|\nabla f|^2\,\omega^n\right )^\frac{1}{2}\notag.
\end{align}
Let $\epsilon\rightarrow 0$, from (\ref{integral}), we have
\begin{align}
\frac{4(p-1)}{p^2}\,\int_M \,\left|\nabla f^{\frac{p}{2}}\right|^2\,\omega^n\,\leq\, A\,\int_M\, f^p\,\omega^n.
\end{align}
Using the Sobolev inequality, we have
$$\frac{4(p-1)}{p^2}\,\left[\frac{1}{C_S}\,\left(\int_M \,f^{\frac{np}{n-1}}\,\omega^n\right)^\frac{n-1}{n}\,-\,\int_M\, f^p\,\omega^n\right]\,\leq\, A\,\int_M \,f^p\,\omega^n.$$
Thus for $\tau =\frac{n}{n-1}$,
 we obtain
\begin{align}\label{lp}
|f|_{L^{p\tau}}\,\leq\, (A\,p\,C_S)^{\frac{1}{p}}\,|f|_{L^p},~~~ p\,\geq\, 2.
\end{align}
Putting $p_k=2\tau^{k-1}$, we get
$$|f|_{L^{p_k}}\,\leq\, (A\, p_{k-1}C_S)^{\frac{1}{p_{k-1}}}\,|f|_{L^{p_{k-1}}}.$$
Taking $k\rightarrow \infty$, we have
$$|f|_{L^\infty}\,\leq\, A^{\frac{n}{2}}\,C\,|f|_{L^2},$$
where $C = C(C_S,n)$ depends only on $C_S$ and $n$. Taking $p=2$ in (\ref{lp}), using the H\"{o}lder inequality
$$|f|^{\frac{n+1}{n}}_{L^2}\,\leq\, |f|_{L^{2\gamma}}\,|f|^{\frac{1}{n}}_{L^1},$$ we have
$$|f|_{L^2}\,\leq\, C \,A^\frac{n}{2}\, |f|_{L^1}.$$ The lemma is proved.
\end{proof}
Since $L$ is $\mathbb{Q}$-line bundle, we can choose an integer $l_0$ such that $L^{l_0}$ is a line bundle and take a hermitian metric on $L^{l_0}$ such that $$R(h_0)\,=\,l_0\,\omega.$$
In the sequel, we always take $l$-th power of $L$ for $l=ml_0$ being a multiple of $l_0$.
So for simplicity, we can think that there is a hermitian metric $h$ on $L$ with its curvature $R(h)$ being $\omega$.
Now we give the gradient estimate of holomorphic sections for conic metrics.
\begin{lem}\label{gradient}
For $s\in H^0(M, L^l)$, we have
\begin{align}\label{section}
|s|_{h}\,+\,l^{-\frac{1}{2}}\,|\nabla s|_{\omega\otimes h}\,\leq\, C\,l^{\frac{n}{2}}\,\left(\int_M\,|s|^2\,\omega^n\right )^{\frac{1}{2}},
\end{align}
where $C$ is again a constant depending only on $n$ and the Sobolev constant $C_S$ of $(M,\omega)$.
\end{lem}
\begin{proof}
In $M\setminus \bigcup D_i$, it holds that
$$\Delta  |s|^2_{h}\,\geq \,-\,n l\,|s|^2_{h}$$
and
$$\Delta |\nabla s |_{\omega\otimes h}^2\,\geq\, -\,l(n+2)\,|\nabla s |_{\omega\otimes h}^2.$$
Then \eqref{section} follows from Lemma \ref{moser}.
\end{proof}
We also need H\"{o}rmander's $L^2$ estimate.
\begin{lem}\label{partial-bar}
For any  $v\in C^\infty(\Gamma(M,L^l))$, there exists a solution $\sigma\in C^\infty(\Gamma(M,L^l))$ such that
$\bar{\partial}\sigma=\bar{\partial}v$ with property:
\begin{align}\label{L^2}
\int_M\,|\sigma|^2\,\omega^n\,\leq\, l^{-1}\,\int_M\,|\bar{\partial}v|^2\,\omega^n.
\end{align}
\end{lem}

\begin{proof}
In $M\setminus \bigcup_{i=1}^k D_i$, we have the Weitzenb\"{o}ch formula
$$\Delta_{\bar\partial}\theta\,=\,{\bar\nabla}^*\bar\nabla \theta\,+\,{\omega}\,(\theta,\cdot)\,+\,l\,\theta$$ for $\theta\in \Lambda^{0,1}(L^{l})$. Let $\gamma_\epsilon$ be the cut-off function supported in $M\setminus \bigcup_{i=1}^k D_i$ as above.

Multiplying both sides of the above identity by $\gamma_\epsilon\theta$ and integrating by parts, we have
$$\int_M\, \gamma_\epsilon\langle \Delta_{\bar\partial}\theta, \theta\rangle\, \omega^n\,=\,\int_M\,\left (l \gamma_\epsilon|\theta|^2
+ \gamma_\epsilon(\omega\otimes h_j)(\theta,\theta)+\gamma_\epsilon|\bar\nabla\theta|^2- \langle (\bar\nabla \gamma_\epsilon,\bar\nabla \theta)_{\omega},\theta\rangle_{h}\right )\,\omega^n.$$
Taking $\epsilon\rightarrow 0$, noticing that
$$\int_M \langle (\bar\nabla \gamma_\epsilon,\bar\nabla \theta)_{\omega},\theta\rangle_{h}\,\leq\, |\theta|_{L^\infty}|\bar\nabla\theta|_{L^\infty}\,\left(\int_M|\nabla \gamma_\epsilon|^2\omega_j^n\right)^{\frac{1}{2}}\rightarrow 0,$$
we get
$$\int_M \langle \Delta_{\bar\partial}\theta, \theta\rangle \,\omega^n \,\geq \,l\int_M|\theta|^2\omega^n.$$
Since $\Lambda^{p,q}(L^l)$ is dense in $L^2(L^l\otimes \Omega^{p,q})$, we know that $\Delta_{\bar\partial}$ is an invertible operator. Putting $\sigma=\bar\partial^*\Delta_{\bar\partial}^{-1}\bar\partial v$, the lemma is proved.
\end{proof}

Like the eigenvalue estimate for smooth metrics, we have
\begin{lem}\label{eigenvalue}
The first eigenvalue of $(M,\omega)$ is at least $t$.
\end{lem}
\begin{proof}
By Theorem \ref{ricci}, there is a smooth approximation $\omega_j$ of $\omega$. The first eigenvalues $\lambda_1^j$ of $\omega_j$ are uniformly bounded and the first eigenfunctions $u_j$ are uniformly Lipschitz. By subsequence, $u_j$ converges to an eigenfunction $u$ on $(M,\omega)$ with eigenvalue
\begin{align}\label{eigen}
\lambda\,=\,\lim_{j\rightarrow\infty} \lambda_1^j.
\end{align}
If $\lambda > \lambda_1$, we will have a function $v$ such that
$$\int_M \,|\nabla v|^2\,\omega^n\,=\,\lambda_1,\,\,\,\,\,\, \int_M v^2\omega^n\,=\,1.$$
It follows from the smooth convergence that
$$\lim_{j\rightarrow \infty} \int_M \,|\nabla v|^2\,\omega_j^n\,=\,\lambda_1,~~~~ \lim_{j\rightarrow \infty}\int_M\, v^2\omega_j^n\,=\,1.$$ So for any $\epsilon>0$, we can find an eigenfunction $v_j$ with eigenvalue not bigger than $\lambda_1+\epsilon$ when $j$ is large enough. It contradicts to (\ref{eigen}). So we know that the first eigenfunction $u$ is the first eigenfunction. In particular, we have
$$|\nabla u|\,\leq\, C, $$
where $C=C(n,C_S)$ is some constant depending on only $n$ and $C_S$.

Now by the Weitzenb\"{o}ch formula, we have
$$\Delta_{\bar\partial} \bar\partial u\,=\,\bar\nabla^* \bar\nabla \bar\partial u\,+\,{\rm Ric}\,(\bar\partial u,).$$
Let $\gamma_\epsilon$ be the cut-off function as above and multiply both sides of by $\gamma_\epsilon^2\bar\partial u$, we have
$$\int_M\gamma^2_\epsilon \langle\Delta \bar\partial u,\bar\partial u\rangle\,\omega^n\,=\,
\int_M\,\langle\bar\nabla^* \bar\nabla \bar\partial u, \gamma^2_\epsilon\bar\partial u\rangle\,\omega^n\,+\,t\int_M\gamma^2_\epsilon|\bar\partial u|^2\,\omega^n.$$
Since $\omega$ is a conic K\"{a}hler-Einstein metric,
Notice that
$$\int_M\gamma^2_\epsilon \langle\Delta \bar\partial u,\bar\partial u\rangle\,\omega^n\,=\,\int_M\gamma^2_\epsilon \langle\bar\partial \Delta u,\bar\partial u\rangle\,\omega^n\,=\,\lambda_1\int_M\gamma^2_\epsilon|\bar\partial u|^2\,\omega^n.$$
Integration by parts, we have
$$\int_M\langle\bar\nabla^* \bar\nabla \bar\partial u, \gamma^2_\epsilon\bar\partial u\rangle\,\omega^n\,=\,\int_M\gamma^2_\epsilon\langle \bar\nabla \bar\partial u,\bar\nabla \bar\partial u\rangle\,\omega^n\,+\,2\int_M \langle \gamma_\epsilon
\bar\nabla \bar\partial u,\bar\nabla \gamma_\epsilon\otimes\bar
\partial u\rangle\,\omega^n.$$
Since
$$2|\int_M \langle \gamma_\epsilon\bar\nabla \bar\partial u,\bar\nabla \gamma_\epsilon\otimes\partial u\rangle\,\omega^n|\,\leq\, \eta\int_M\gamma^2_\epsilon\langle \bar\nabla \bar\partial u,\bar\nabla \bar\partial u\rangle\,\omega^n+ \frac{C}{\eta}\int_M|\nabla \gamma_\epsilon|^2\omega^n,$$
we get $$(1-\eta)\int_M\gamma^2_\epsilon\langle \bar\nabla \bar\partial u,\bar\nabla \bar\partial u\rangle\,\omega^n\,\leq\, (\lambda-t)\,\int_M\gamma^2_\epsilon|\bar\partial u|^2\,\omega^n+\frac{C}{\eta}\int_M|\nabla \gamma_\epsilon|^2\,\omega^n.$$
Taking $\epsilon\rightarrow 0$, and then $\eta\rightarrow 0$, we obtain
\begin{align}\label{bochner}
\int_M\,\langle \bar\nabla \bar\partial u,\bar\nabla \bar\partial u\rangle\,\omega^n\,\leq\, (\lambda_1-t)\int_M|\bar\partial u|^2\,\omega^n.
\end{align}
So $\lambda_1\geq t.$
\end{proof}
\section{Compactness and Partial $C^0$-estimate}
Let $L$ be a $\mathbb{Q}$-line bundle on $M$. Define
$$\mathcal{M}(K,A)\,=\,\{\,(M,\omega)\,|\,\omega\,\in\, 2\pi\, c_1(L),~~ {\rm Ric}(\omega)\,\geq\, K\,\omega,~~ {\rm diam}(M,\omega)\,\leq\, A\,\}.$$
The volume of any metrics in $\mathcal{M}(K,A)$ is a fixed constant $V$.

Let $D_i$ be the smooth divisors in $M$ as above and $(M,\omega_j, D)$, where $D$ denotes the divisor $\Sigma_{i=1}^k(1-\beta^i_j)D_i$, be
a sequence of conic K\"{a}hler-Einstein metrics satisfying:
$${\rm Ric}(\omega_j)\,=\,t_j\,\omega_j\,+\,2\pi\,\Sigma_{i=1}^k\,(1-\beta^i_j)\,D_i, ~~~~\omega_j\,\in\, 2\pi c_1(L),$$
where both $t_j$ and $\beta^i_j$ are in $[\mu,1]$ for some $\mu > 0$.
Assume that $(M,\omega_j,D)$ converges to a metric space $(X,d)$ in the Gromov-Hausdorff topology. From Theorem \ref{ricci}, we know that $$(X,d)\,\in\,\overline{\mathcal{M}}(K,\frac{\pi}{\sqrt{\mu}}).$$

We want to show that the limit is a normal variety. The key is to establish the partial $C^0$ estimate for $\omega_j$. Let us make it more precise.
For any conic K\"ahler metric $\omega$, choose a hermitian metric $h$ on $L$ such that $\text{R }(h)=\omega$.
Then for a large integer $\ell$ such that $L^{\ell}$ is a line bundle, we have the $L^2$ metric on $H^0(M,L^{\ell})$ :
$$\langle s_1,s_2\rangle \,=\,\int_M\,\langle s_1,s_2\rangle_{h^{\otimes \ell}}\,\omega^n,~~~~s_i\,\in\, H^0(M,L^{\ell} ).$$
\begin{defi}
For any orthonormal basis $\{s_i\}_{0\leq i\leq N}$ of $H^0(M,L^{\ell})$, we define a Bergman kernel as
\begin{align}
\rho_\ell(M,\omega)(x)\,=\,\Sigma_{i=0}^N \,|s_i|^2_{h^{\otimes \ell}}(x).
\end{align}
\end{defi}
The partial $C^0$ estimate we want to establish for $(M,\omega_j)$ is stated as follows: There exists some large integer $\ell$ and some positive constant $c>0$, which may depend on $l$, such that
$$\rho_\ell(M,\omega_j)\,\geq\, c.$$
We will follow the strategy in \cite{Ti15} to prove this.

At first, the Cheeger-Colding theory (\cite{CC1}) assures the existence of tangent cones $C_x=C_x(X,d)$ as a metric cone, which may not be unique, at any $x\in X$. Then we have a decomposition
$$X\,=\, \mathcal{R}\cup  \mathcal{S},$$
where $\mathcal{R}$ consists of all $x\in X$ which has a tangent cone $C_x$ isometric to the Euclidean space $\RR^{2n}$ and $\mathcal{S}$ is the complement, referred as the singular set. Also every tangent cone has a similar decomposition. Furthermore, it follows from \cite{CC1} and \cite{CCT}
\begin{lem}
There is a stratification of the singular set $\mathcal{S}$ of $X$:
$$\mathcal{S}\,=\,\bigcup_{k\geq 1}\mathcal{S}_{2n-2k},$$
where $\mathcal{S}_{2n-2k}$ consists of all $x$ for which at least one tangent cone splits an Euclidean factor $\CC^{n-k}$ and no tangent cones split Euclidean factor
of dimension greater than $2n-2k$. We further have
$$\dim \mathcal{S}_{2n-2k}\,\leq\, 2n-2k.$$
Also every tangent cone $C_x$ of $X$ has the stratification of the same property.
\end{lem}
It follows from volume convergence in \cite{Co} and Anderson's result in \cite{An}.
\begin{lem}\label{angle}
If there exists $\bar\beta<1$ such that $\beta^i_j\leq \bar\beta$, then the regular part $\mathcal{R}$ is open, and the distance is induced by a smooth metric $\omega_\infty$ on $\mathcal{R}$. Moreover, for any compact set $K\subset\mathcal{R}$, there are $K_j\subset (M,\omega_j)$, such that $(K_j,\omega_j)$ converges to $(K,\omega_\infty)$ in the smooth topology.
\end{lem}

For tangent cones, we have
\begin{lem}\label{beta1}
Assume that all $\beta^i_j\leq \bar\beta<1$. Let $C_x$ be a tangent cone of $(X,d)$ at $x$:
$$(C_x,\omega_x)\,=\,\lim_{a\rightarrow \infty}(X,r_a^{-1}\,d,x).$$
Then the regular part $\mathcal{R}(C_x)$ is open, and the distance is induced by a smooth metric $\omega_x$ on $\mathcal{R}(C_x)$.
Moreover for any compact set $K\subset\mathcal{R}(C_x)$, there are $K_j\subset (\mathcal{R},\omega_\infty)$, such that $(K_a,r^{-2}_a\,\omega_\infty)$
converges to $(K,\omega_x)$ in the smooth topology.
\end{lem}

In general, there exists  $l\in [1,k]$ so that $\beta_j^i(1\leq i\leq l)$ approaches $1$ and $\beta_j^i(l+1\leq i\leq k)$ is not greater than $\bar \beta$. From the volume comparison, there exists $\delta=\delta(n,\bar\beta)$ such that for the metric $\omega_j$, if we have
$$\text{vol }(B_r(p_j,\omega_j))\,\geq \,(1-\delta)\,V_{2n}\,r^{2n},$$
where $B_r(p,\omega)$ denotes the metric ball of $(M,\omega)$ with radius $r$ and center $p$ and $V_{2n}$ denotes the volume of the unit Euclidean ball,
then
$$B_r(p_j,\omega_j) \cap \bigcup_{i=l+1}^k D_i\,=\,\emptyset.$$
We denote by $B_r(p_\infty,d)$ the limit of $B_r(p_j,\omega_j)$ in the Gromov-Hausdorff topology.
We need the following regularity theorem for weak limits of conic K\"ahler-Einstein metrics (possibly singular).
\begin{prop}\label{lr}
There exists $\delta=\delta(n,\bar\beta), \rho=\rho(n)$ such that if
$${\rm vol}(B_r(p_\infty,d))\,\geq\, (1-\delta)\,V_{2n}r^{2n},$$
then $B_{\rho r}(p_\infty,d)$ is biholomorphic to an open subset in $\CC^n$ and the metric $d$ on $B_{\rho r}(p_\infty,d)$ is induced by a smooth K\"ahler-Einstein metric $\omega_\infty$, moreover, the Gromov-Hausdorff limits $D_i^\infty $ of $D_i$ ($1\le i\le l$) are divisors in $B_{\rho r}(p_\infty,d)$.
\end{prop}
Since we only know a prior that $\omega_j$ converges to $d$ in the weak topology, i.e., the Gromov-Hausdorff topology, we can not apply the Cheeger-Colding theorey directly to prove the smoothness of $B_{\rho r}(p_\infty,d)$ and $d$. There are two approaches to establishing such a regularity: one is to use the Ricci flow method as in \cite{TW16} (also see Appendix B in \cite{Ti15}), the other one works only in the complex case and is to use the technique of partial $C^0$-estimate. This second approach was first used in \cite{CDS} and will be used to prove Proposition \ref{lr}. The first approach can be also applied after we
extend certain results on the ordinary case of Ricci flow to the conic case.

Denote the scaled metrics
$$\tilde\omega_j\,=\,r^{-2}\omega_j~~~~{\rm and}~~~~\tilde\omega_\infty\,=\,r^{-2}\omega_\infty.$$
We will denote by $\tilde B_\rho(p_j)$ and $\tilde B_\rho(p_\infty)$ the metric balls of radius $\rho$ in $(M,\omega_j)$ and $(X,\tilde d)$, where $\tilde d = r^{-1} d$.
Let $L_{Euc}$ be the trivial bundle over over the unit ball $B^{2n}_1(0)\subset \RR^{2n}$ with curvature $\omega_{Euc}$ being the K\"ahler form of the Euclidean metric.
Let $m(W)$ be the Minkowski measure of a subset $W$ with respect to a metric $\tilde d$, that is,
$$m(W)\,=\,\inf\{\,A\,|\,\forall r>0,\, \exists \,A r^{2-2n}~{\rm metric~ balls~ of~ radius~} r ~{\rm covering} \,W\,\}.
$$

\begin{prop}\label{map}
For any $A, \bar\beta>0$, there are $\rho=\rho(A), \delta=\delta(A,\bar\beta) >0$ with the following property.
First Assume
\begin{equation}\label{prop-map-0}
vol(\tilde B_1(p_\infty))\,\geq\, (1-\delta)V_{2n}.
\end{equation}
Next we assume that there is subset $W^\infty\subset \tilde B_1(p_\infty)$ with $m(W^\infty)\leq A$,
such that $(\tilde B_1(p_\infty)\setminus W^\infty,\tilde \omega_\infty) $ is a smooth K\"ahler manifold and $D_i^\infty (1\leq i\leq l)$ are divisors in
$\tilde B_1(p_\infty)\setminus W^\infty$.
Then for $j$ sufficiently large, we have
\vskip 0.1in
\noindent
1. There is a local K\"ahler potential $\varphi_j$ for $\tilde \omega_j$ on the ball $\tilde B_\rho(p_j)$ such that
\begin{equation}\label{prop-map-1}
|\varphi_j|\,\leq\, C(n)\,\rho^2.
\end{equation}

\vskip 0.1in
\noindent
2. There is a holomorphic map $F_j: \tilde B_\rho(p_j)\,\mapsto \,\mathbb{C}^n$ which is a homeomorphism to its image and satisfies:
\begin{eqnarray}\label{prop-map-2}
&|\nabla F_j|_{\tilde \omega_j}\,\leq\,  C(n)\\
&\label{prop-map-3} B_{\frac{3\rho }{4}}^{2n}(0)\,\subset\,F_j(\tilde B_\rho(p_j)) \,\subset\,B_{\frac{5\rho}{4}}^{2n}(0).
\end{eqnarray}
\end{prop}
By Anderson's result, we know that $\tilde \omega_j$ converges to $\tilde\omega_\infty$ in the smooth topology on any compact subset
outside $W^\infty$ and $D_i^\infty$ for $i=1,\cdots,l$, moreover, for any $\epsilon_1,\eta>0$, there exists $\delta_0=\delta_0(\epsilon_1, \eta)$ such that if \eqref{prop-map-0} holds for some $\delta\leq \delta_0$, there is injection
$$\phi :\tilde B_{\frac{3}{4}}(p_\infty)\setminus T_{\epsilon_1}(W^\infty) \,\rightarrow\, B_1^{2n}(0)$$
such that
\begin{equation}\label{eq:t1}
|\phi^*\omega_{Euc}\,-\,\omega_\infty|_{C^3}\,\leq\, \eta,
\end{equation}
where $\omega_{Euc}$ denotes the Euclidean metric on $\CC^n$.

We denote by $L_\infty$ the limit of $L$ on $\tilde B_{\frac{1}{2}}(p_\infty)\setminus W^\infty \cup \left(\cup_{i=1}^l D_i^\infty\right)$. Since $\omega_j$ are conic K\"ahler-Einstein metrics with K\"ahler class $2\pi c_1(L)$,
they induce natural hermitian metrics on $L$ which converge to a hermitian metric on $L_\infty$ whose curvature is $\omega_\infty$.
Moreover, since $\omega_\infty$ is smooth outside $W^\infty$, the line bundle
$L_\infty$ can be extended to $\tilde B_{\frac{1}{2}}(0)\setminus W^\infty$. We start with the following lemma:

\begin{lem}\label{circle}
Given any $\epsilon_1>0, \eta>0, A>0$, there exist $\tilde{\epsilon}_1, \tilde{\eta},\delta_1>0$ such that if
$$m(W^\infty)\,\le \,A,\,\,\,\,\,\,vol(\tilde B_1(p_\infty)\,\geq \,(1-\delta_1)\,V_{2n}$$
and for any circle $\gamma$ of radius $\tilde{\epsilon}_1$ in $\tilde B_1(p_\infty)\setminus T_{\tilde{\epsilon}_1}(W^\infty)$, we have
$$|H_\gamma - Id|\,\le\, \tilde \eta,$$
where $H_\gamma$ denotes the holonomy of $L_\infty^\vee \otimes \phi^* L_{Euc}$ around $\gamma$,
there is a norm 1 section $e$ of $L_\infty^\vee \otimes \phi^* L_{Euc}$ over $\tilde B_1(p_\infty)\setminus T_{\epsilon_1}(W^\infty)$
with $|D e|\leq \eta$.
\end{lem}
\begin{proof}
This is a pretty trivial fact. For the readers' convenience, we include a standard proof by contradiction. One can also give an effective proof by estimating holonomy in terms of curvature. If Lemma is false, then for certain $\epsilon_1>0, \eta>0, A>0$, we can find $\delta_k\mapsto 0$, $\eta_i\mapsto 0$
and $W_{ik} \subset \tilde B_1(p_{ik})$ satisfying: $$m(W_{ik})\,\le \,A,\,\,\,\,\,\,vol(\tilde B_1(p_{ik})\,\geq \,(1-\delta_k)\,V_{2n}$$
and the holonomy of $L_\infty^\vee \otimes \phi^* L_{Euc}$ around any circle of radius $\frac{1}{i}$ inside $B_1(p{ik})\setminus W_{ik} $
differs from $Id$ no more than $\eta_i$, but there is no norm 1 section $e$ with $|D e|\leq \eta$. Since $\omega_\infty$ can be approximated by smooth metrics with Ricci curvature bounded from below, by Letting $k$ go to $\infty$, we can get $W_{i}$ inside $B_1^{2n}(0)$ with $m(W_{i})\le A$ and a flat bundle $L_{i}$ whose holonomy around any circle of radius $\frac{1}{i}$ inside $B_1^{2n}(0)\setminus W_{i} $ is less than $\eta_i$. It follows that outside the limit $W'$ of $W_{i}$ there is a flat bundle $L'$. Also we have $m(W')\leq A$, so any loop $\gamma\subset B_1^{2n}(0)\setminus W'$ bounds a disc $\Sigma$ which intersects with $W'$ at finitely many, say m,  points, consequently, $\gamma$ is homologous to m circles of arbitrarily small radius. By our choice of $W_{i}$, the holonomy of a circle in $
B_1^{2n}(0)\setminus W'$ can be as close to $Id$ as one wants so long its radius is sufficiently small. Since $L'$ is flat, it follows that the holonomy of $L'$ around $\gamma$ is trivial. So $L'$ must be a trivial bundle with a trivial connection. Then there is a norm 1 section $e_i$ of $L_{i}$ for $i$ sufficiently large
such that $|D e_i|\leq \eta$. It is a contradiction. The lemma is proved.
\end{proof}

\begin{lem}\label{stokes}
Given any $\tilde\epsilon_1>0$, there exists $\delta_2>0$, such that if
$$vol(\tilde B_1(p_\infty)\,\geq\, (1-\delta_2)\,V_{2n},$$
then the holonomy of $L_\infty$ around any circle of radius $\tilde{\epsilon}_1$ in $\tilde B_1(p_\infty) \setminus W^\infty$ is less than $10\,\tilde{\epsilon}^2_1$.
\end{lem}
\begin{proof}
For a circle of radius $\tilde{\epsilon}_1$, after a slight perturbation we can assume that it has no intersection with $\bigcup_{i=1}^l D_i^\infty$. By the smooth convergence outside $W^\infty\bigcup_{i=1}^l D_i^\infty$, it suffices to estimate the holonomy for $L$ over $\tilde B_1(p_j)$. By the volume convergence result
in \cite{Co} and Appendix 1 in \cite{CC1}, for $\delta_2$ small, $\tilde{B}_1(p_j)$ is homeomorphic to $B_1^{2n}(0)$ for $j$ sufficiently large. So the connection on $L$ is defined by a 1-form $\varphi$. Pulling back the coordinates functions of $B_1^{2n}(0)$ to $\tilde{B}_1(p_j)$, we can get the splitting functions $(f_1,f_2,...,f_{2n})$ satisfying:
\begin{align}\label{orthogonal-5}
 &\frac{1}{\text{vol}(\tilde{B}_1(p_j))}\,\int_{\tilde{B}_1(p_j)}
 \,\Sigma_i\,|{\rm Hess }\,f_i|^2+\Sigma_{i\neq j}|\langle\nabla f_i,\nabla f_j\rangle|\notag\\
 &+\frac{1}{\text{vol}(\tilde{B}_1(p_{\infty}))}\,\int_{\tilde{B}_1(p_\infty)}\,\Sigma_i(|\nabla f_i|-1)^2 \,<\,\Psi(\delta_2).
 \end{align}
 As in [CCT], there are functions $\bf r, \bf{u}$ satisfying:

 \begin{eqnarray}\nonumber
\left\{                  
\begin{array}{lllll}
\Delta {\bf r}^2=2n\\
 {\bf r}^2|_{\partial \tilde{B}_{\frac{1}{2}}(p_j)}\,=\,\frac{1}{4}.
\end{array}            
\right.
\end{eqnarray}
and
\begin{eqnarray}\nonumber
\left\{                  
\begin{array}{lllll}
\Delta {\bf u}^2\,=\,2n-2\\
 {\bf u}^2|_{\partial \tilde{B}_{\frac{1}{2}}(p_j)}\,=\,{\bf r}^2\,-\,\sum_{j=1}^{2n-2}\,f_j^2.        \end{array}
 \right.
\end{eqnarray}

 Denoting $\Phi=(f_1, f_2,...,f_{2n-2})$ as in Theorem 2.63 in \cite{CCT},
 we know that away from a subset of measure $\Psi(\delta)$, it holds that
 $$\int_{\Sigma_{z,{\bf u}}}\omega=\frac{u^2}{2}+\Psi(\delta_2),$$
 where $\Sigma_{z, u}=\Phi^{-1}(z)\cap\{{\bf u}\leq u\}$.
 For $\delta_2$ small, integrating by parts, we get
 $$\int_{\partial \Sigma_{z, \tilde\epsilon_1} }\,\varphi \,=\,\int_{\Sigma_{z,\tilde\epsilon_1}}\, d\varphi\,=\,\int_{\Sigma_{z,\tilde\epsilon_1}}\,\omega\,\leq\, 10\,\tilde\epsilon_1^2.$$
\end{proof}

\begin{lem}\label{holo}
For any $\epsilon_1, \eta>0$, there exists $\delta>0$ such that if
$$vol(\tilde B_1(p_\infty))\,\geq\, (1-\delta)\,V_{2n},$$
then there is a hermitian bundle map $\psi$ from $L_\infty$ to $L_{Euc}$ over $\tilde B_{\frac{1}{2}}(p_\infty)\setminus T_{\epsilon_1}(W^\infty)$
lifting $\phi$ such that $|D\psi|\,\leq\,\eta$.
\end{lem}
\begin{proof}
We can choose $\epsilon_1$ in Lemma \ref{circle} such that $\tilde\epsilon_1^2\leq \tilde \eta$. Now by Lemma \ref{circle} and Lemma \ref{stokes}, we only need to take $\delta=\min\{\delta_1, \delta_2\}$.
\end{proof}
Now we are going to prove Proposition \ref{map}.
\begin{proof} Lemma \ref{holo} gives a bundle map $\psi: L_\infty\mapsto L_{Euc}$
over $\tilde B_{\frac{1}{2}}(p_\infty)\setminus T_{\epsilon_1}(W^\infty)$ lifting $\phi$ such that $|D\psi|\,\leq\,\eta$.
Since $m(T^\infty)\leq A$, by using a well-known property of capacity, we get a cut-off function $\gamma_{\bar\epsilon}$ for any given $\bar\epsilon>0$ satisfying:
Its support is contained in $T_{\epsilon_1}(W^\infty)$ for $\epsilon_1=\epsilon_1(\bar\epsilon,A)$ and its gradient has $L^2$ norm smaller than $\bar\epsilon$.
We choose $\delta_0=\delta_0(\epsilon_1, \eta)$ such that \eqref{eq:t1} holds. As in the proof of the partial $C^0$ estimate in Section 5 in \cite{Ti15}, choosing
$\bar\epsilon$ and $\eta=\eta(\bar\epsilon)$ (also see the proof in Lemma \ref{cut-off}), we can find a multiple $\ell=\ell(A)$ and holomorphic sections $s_j, s^{\alpha}_j$ of $L^\ell$ from approximate holomorphic sections
$$\tau\,=\,\psi^{-1}(\gamma_{\bar\epsilon} e),~~~ \tau^{\alpha}\,=\,\psi^{-1}(\gamma_{\bar\epsilon} z^\alpha ),~~~\alpha \,=\,1,\cdots,n,$$
where $e$ is the trivial section and $z^\alpha$ are the coordinates of $\CC^n$. Then we define $\rho=\frac{1}{\ell}$ and the map $F_j$ by
$$F_j\,=\,\left(\frac{s^{1}_j}{s_j},..., \frac{s^{n}_j}{s_j}\right).$$
\end{proof}

Now we use induction on the volume $\bar A=\sum_{i=1}^l vol(D_i\cap \tilde B_1(p_j))$ to prove Proposition \ref{lr}. Define the volume density function by
\begin{equation}\label{eq:t2}
V(x,s)\,=\,r^{2-2n}\sum_{i=1}^l \,vol(D_i\cap B_r(x,\tilde \omega_j)).
\end{equation}
Note that it depends on $\omega_j$ though we do not make it explicitly for simplicity.

We start with the following lemma:
\begin{lem}\label{induc}
There are $\delta_1=\delta_1(n,\bar\beta), \kappa=\kappa(n)$ such that if
$$\text{vol }(\tilde B_1(p_j))\,\geq\, (1-\delta_1)\,V_{2n}, $$
then $V(q,\bar r)\geq \kappa$ for all $\bar r$ and $q\in \bigcup_{i=1}^l D_i$
such that $B_{\bar r}(q,\tilde\omega_j)\subset \tilde B_{\frac{1}{2}}(p_j)$.
\end{lem}
\begin{proof}
Let $\delta=\delta(2^{2n},\bar\beta)$ as in Proposition \ref{map}. Let $(x_j,s_j)$ be the point in $\tilde B_{\frac{1}{2}}(p_j)$ attaining the minimum of $V(x,s)$ an $\bar\omega_j= s_j^{-2}\,\tilde\omega_j$. Clearly, for $B_j=B_1(x_j,\bar\omega_j)$, the geodesic ball with radius $1$ of $\bar\omega_j$,
we have
$$m(B_j\bigcap \bigcup_{i=1}^l\,D_i)\,\leq\, 2^{2n}.$$
Let $B_\infty$ be the limit of $B_j$ and $W^\infty$ be $B_\infty \bigcap\bigcup_{i=1}^l D_i^\infty$,
then $\bar \omega_j$ converges to $\bar \omega_\infty$ smoothly in $B_1(x_\infty,\bar\omega_\infty)\setminus W^\infty$, where $x_\infty = \lim x_j$.
Applying Proposition \ref{map}, we get a holomorphic map $F_j$ with bounded image. Using the monotonicity of corresponding $V(x,s)$ in the Euclidean space, we have
$V(x_j,s_j)\geq \kappa(n)$. The lemma is proved.
\end{proof}

The following is a known consequence of the well-established regularity theory.
\begin{lem}\label{chart}
For any $C, A>0$, suppose that
$$\sum_{i=1}^l \,vol(D_i\bigcap \tilde B_1(p_j))\,\leq\, A$$
and there is a holomorphic map $F_j:\tilde B_1(p_j)\rightarrow \mathbb{C}^n$ which is a homeomorphism to its image
and satisfies
$$(F_j)_*\omega_j\,=\,\sqrt{-1}\,\partial \bar\partial \,\varphi_j\,\,\,\,{\rm with}\,\,\,|\varphi_j|\,\leq\, C \,\,\,{\rm  and}\,\,\,|\nabla F_j|\,\leq\, C.$$
Then the limiting metric $\tilde d$ on $\tilde B_{\frac{1}{2}}(p_\infty)$ is induced by a smooth K\"ahler-Einstein metric $\omega_\infty$ and
$D_i^\infty$ are divisors in $\tilde B_{\frac{1}{2}}(p_\infty)$.
\end{lem}
\begin{proof}
Using the holomorphic map $F_j$, we can write the following equations:
$$(\sqrt{-1}\partial\bar\partial \,\varphi_j)^n\,=\,e^{-\varphi_i}\,\Pi_{i=1}^l\,|f_i|^{2(\beta_j^i-1)}\,|U|^2,$$
where $f_i$ is the defining function of $D_i$ and $U$ is a nonvanishing holomorphic function. Our assumptions say that for some constant $C$,
$$|\varphi_j|\,\leq \,C, \,\,\,\,\,\omega_{Euc}\,\leq\, C\,\sqrt{-1}\,\partial\bar\partial \,\varphi_j.$$
Because the volume of $D_i$ is bounded with respect to any $\tilde \omega_j$,
for each $i$, there is a limit $D_i^\infty$ of $D_i$. Furthermore, since $\beta_j^i$ tend to $1$, by taking a subsequence if necessary, $\varphi_j$ converge to
a bounded $\varphi_\infty$ such that
$$ C^{-1}\, \omega_{Euc}\,\leq\, \sqrt{-1}\,\partial\bar\partial \,\varphi_\infty\,\le \,C\, \omega_{Euc}.$$
Then, it follows from the standard regularity theory for complex Monge-Ampere equations (see \cite{Bl}) that
the limit $\omega_\infty$ is a smooth K\"ahler-Einstein metric.
\end{proof}

\begin{proof}[Proof of Proposition \ref{lr}]
Choose $\delta$ to be $\delta_1$ in Lemma \ref{induc}. If $\bar A\leq \kappa$, the proposition follows from Lemma \ref{induc} and Anderson's result.
Now we assume that the proposition holds for $\bar A\leq 2^m \kappa$, we want to prove it still holds for $\bar A\leq 2^{m+1} \kappa$.
Denote
$$Z^j_s\,=\,\left \{q\in \tilde B_{\frac{1}{2}}(p_j)\,|\, V(q,s)\,\geq\, \frac{\bar A}{2}\right \}$$
and
$$W_s^\infty\,= \,\lim_{j\rightarrow\infty}Z^j_s\,\,\,\,\,{\rm and}\,\,\,\,W^\infty\,=\,\bigcap_{s\leq 1}\,W_s^\infty.$$
It follows from a standard covering argument that $m(W^\infty)\leq 2^{2n}$.
By induction, $\omega_\infty$ is smooth in $\tilde B_{\frac{1}{2}}(p_\infty)\setminus W^\infty$ and $D_i^\infty$ ($1\leq i\leq l$) are divisors.
Now by Proposition \ref{map}, there is a holomorphic map $F_j$ from
$\tilde B_\rho(p_j)$ into $\mathbb{C}^n$ needed in Lemma \ref{chart}, consequently, $\omega_\infty$ is smooth and $D_i^\infty $($1\leq i\leq l$)
are divisors in $\tilde B_{\frac{\rho}{2}}(p_\infty)$. The proposition is proved.
\end{proof}

As a direct consequence of Proposition \ref{lr}, we have the following regularity for the limit space $(X,d)$.
\begin{lem}\label{smtangent}
The regular part $\mathcal{R}$ is open, and the distance function $d$ is induced by a smooth metric $\omega_\infty$ on $\mathcal{R}$ and $D_i^\infty$ are closures of divisors in $\mathcal{R}$. Moreover for any compact set $K\subset\mathcal{R}\setminus \bigcup_{i=1}^l D_i^\infty$, there are $K_j\subset (M,\omega_j)$, such that $(K_j,\omega_j)$ converges to $(K,\omega_\infty)$ in the smooth topology.
\end{lem}

Similarly for the tangent cones, we have
\begin{lem}\label{smtangent2}
Let $C_x$ be any tangent cone of $(X,d)$ at $x$ given by
$$(C_x,d_x)\,=\,\lim_{j\to \infty}(X,r_j^{-1}\,d,x),$$
where $\{r_j\}$ is a sequence with $\lim r_j=0$.
Then the regular part $\mathcal{R}(C_x)$ is open, and the distance $d_x$ is induced by a smooth metric $\omega_x$ on $\mathcal{R}(C_x)$. Also, $(X,r_j^{-1}\,d,x)$ converges smoothly to $(C_x,d_x)$ over $\mathcal{R}(C_x)$.
\end{lem}
The basic strategy to prove the partial $C^0$ estimate is to pull back a peak section from a tangent cone and then perturb it to a holomorphic section by solving the $\bar\partial$-equation.

On the regular part $\mathcal{R}$ of the limit $(X,d)$, we have a limit line bundle $L_{\mathcal{R}}$. On the regular part $\mathcal{R}(C_x)$ of any tangent cone $C_x$, we have
$$\omega_x\,=\,\sqrt{-1}\,\partial\bar\partial\, \rho_x^2,$$
where $\rho_x$ is the distance to the vertex of $C_x$. Let $L_0$ be the trivial line bundle on $C_x$ equipped with the hermitian metric $e^{-{\rho_x^2}}$.
We have the following lemma (Lemma 5.7 in \cite{Ti15}).
\begin{lem}\label{trivial}
Let $C_x=\lim_{a\to\infty}(X,\sqrt{\ell_a}\,d,x)$ for $x\in (X,d)$, where $\ell_a\in \ZZ$ and $\ell_a\to \infty$, and
$K$ be any compact subset $\mathcal{R}(C_x)$. Then for any $\eta>0$, there exist positive integers $N=N(K,\eta)$ and $a_0$ such that for $a\ge a_0$,
there exist $\ell=\ell(j)\le N$ and a sequence of ${K}_a\subseteq (\mathcal{R},\omega_\infty)$
with property: There are diffeomorphisms $\phi_a: K\mapsto K_a$ and bundle homomorphisms $\psi_a$ lifting $\phi_a$ satisfying:
\begin{align}\label{bundle}
\CD
L_0@>{\psi}_a\,>> \,L_{\mathcal{R}}^{\ell\ell_a}\big|_{{K}_a}\\
  @V  VV @V  VV  \\
  K @>{\phi}_a\,>>\,{K}_a,
\endCD
\end{align}
and
$$\lim_{a\to\infty}\,{\phi}_a^*(\ell\ell_a \omega_\infty)\,=\, \omega_x,\,\,\,\,| D \psi_a|_{\omega_x}\,\le\, \eta   \,\,\,{\rm in}\,\,K,$$
where $|\cdot|_{\omega_x}$ is defined by $\omega_\infty$, $\omega_x$ and corresponding hermitian metrics on $L_0, L_{\mathcal R}$.
\end{lem}

Let ${\mathcal S}_x$ denote the singular set of $C_x$, then we have
\begin{lem}\label{cut-off}
Assume that for any $\bar{\epsilon}>0$, there is a smooth
function $\gamma_{\bar\epsilon}$ on $B_{\bar\epsilon^{-1}}(o,\omega_x)\subset C_x$ satisfying:

\vskip 0.1in
\noindent
1) $\gamma_{\bar\epsilon}\equiv  0$ near ${\mathcal S}_x$ and $\gamma_{\bar\epsilon}(y)\equiv 1$ if ${\rm dist} (y, {\mathcal S}_x) >\bar\epsilon$;

\vskip 0.1in
\noindent
2) $0\le \gamma_{\bar\epsilon}\le 1$ and $|\nabla \gamma_{\bar\epsilon}|\le C = C(\bar\epsilon)$;
\vskip 0.1in
\noindent
3) The integral bound:
$$\int_{  B_{\bar\epsilon^{-1}}(o)} \,|\nabla \gamma_{\bar\epsilon}|^2\,\omega_x^n\,\le \,\bar\epsilon.$$
\vskip 0.1in
\noindent
Then the partial $C^0$ holds at $x$, that is, there exists some large integer $\ell=\ell_x$ and some positive number $c_x$ such that for $j$ sufficiently large and $p\in M$ sufficiently close to $x$,
$$\rho_\ell(M,\omega_j)(p)\,\geq\, c_x.$$
\end{lem}
\begin{proof}
Let $C=C(n,C_s)$ be the constant in (\ref{section}). Choose a small $\delta\leq \frac{1}{8C}$. Put
$$V(x,\delta)\,=\,\{y\in C_x\,|\, \delta\,\leq\, {\rm dist} (y,x)\,\leq\,\delta^{-1},\,{\rm dist} (y, {\mathcal S}_x\,\geq\, \delta\}.$$
We can choose $\delta$ smaller such that $V(x,2\delta)$ is not empty.
Now in $V(x;\delta)$, there exists $C_1$, which may depend on this region, such that if
\begin{align}
\bar{\partial}\sigma\,=\,f,
\end{align}
we have the local estimate for $\sigma$:

\begin{align}\label{elliptic}
|\sigma|_{C^0(V(x;2\delta))}\,\leq\, C_1\,\left(|f|_{C^0(V(x;\delta))}\,+\,\delta^{-2n}\,\int_{V(x;\delta)}\,|\sigma|^2\,\omega_x^n\right).
\end{align}
Choose $\bar\epsilon$ such that $100 C_1\bar\epsilon \leq \delta^{2n+1}$, then by the assumption, we have $\gamma_{\bar\epsilon}$ with property: $\gamma_{\bar\epsilon}\equiv  1$ on $V(x,\delta)$. Let $\zeta$ be a cut-off function $\zeta$ such that $\zeta(t)=0$ for $t\ge 1$ and $\zeta(t)=1$ for $t\le \frac{1}{2}$, then we set
$$\beta(y)\,=\,\zeta(\frac{\bar\epsilon}{d(y,x)})\,\zeta(\bar\epsilon d(y,x))\,\gamma_{\bar\epsilon}(y),\,\,\,\forall y\,\in\, C_x.$$
Then, replacing $\gamma_{\bar\epsilon}$ by $\gamma_{\bar\epsilon/2}$ if necessary, we can have
$$\int_{C_x}\,|\nabla\beta|^2\,e^{-{d^2(x,\cdot)}}\,\omega_x^n\,\leq \,\bar\epsilon.$$

Denote by $K$ the support of $\beta$. By Lemma \ref{trivial}, there are $K_a\subset\mathcal{R}$ converges to $K$ smoothly. Denoting by $e$ the constant section $1$ of $L_0$, then by Lemma \ref{trivial},
we get a smooth section $\tau_a=\psi_a(\beta e)$ of $L_{\mathcal R}^{\ell\ell_a}$ over $\mathcal{R}$. For any $0<\,\eta\leq\,{C_1}^{-1} \delta$, if $a$ is sufficiently large, $\tau_a$ has the following properties:
\begin{align}\label{norm}
\begin{cases}
\,\,\,\|\tau_a\|^2_{C^0(\phi_a(V(x;\delta)\bigcap B_{3\delta}(x)))}\,\geq \,(1-\eta)\,e^{-3\delta^2};&\\
\,\,\,\int_{X}\,|\tau_a|^2\,\omega^n_\infty\,\leq\, (1+\eta)\,\left(\frac{2 \pi}{\ell\ell_a}\right)^n;&\\
\,\,\,|\bar{\partial} \tau_a|\,\leq\,\eta\,\,\,{\rm in}\,\,V(x;\delta);&\\
\,\,\,\int_{X}\,|\bar{\partial}\tau_a|^2 \,\omega^n_\infty\,\leq\, (1+\eta)\,\frac{\bar\epsilon}{(\ell\ell_a)^{n-1}}.
\end{cases}
\end{align}
By Lemma \ref{smtangent}, $D_i^\infty\,(1\leq i\leq l)$ are divisors in $\mathcal{R}$. Then we can choose a suitable cut-off function $\tilde{\gamma}_{\tilde \epsilon}$ supported outside $\bigcup_{i=1}^l D_i^\infty$ such that $\tilde\tau_a=\tau_a\tilde{\gamma}_{\tilde \epsilon}$ satisfies the above estimates with a slightly bigger $\eta$.
By Lemma \ref{smtangent}, there are
$$\phi_j: K_a\,\cap \,supp (\tilde{\gamma}_{\tilde \epsilon})\, \rightarrow \,(M,\omega_j)$$
such that $\phi_j^*\omega_j\rightarrow \omega_\infty$, furthermore, these $\phi_j$ can be lifted to isomorphisms $\psi_j$:
\begin{align}\label{bundle2}
\CD
L_{\mathcal R}^{\ell\ell_a}@>{\psi}_j\,>> \,L^{\ell\ell_a}\\
  @V  VV @V  VV  \\
  K_a\cap \,supp (\tilde{\gamma}_{\tilde \epsilon}) @>{\phi}_j\,>>\,M.
\endCD
\end{align}
Then $v_j=\psi_j(\tilde\tau_a)$ is a smooth section of $L^{\ell\ell_a}$. For $j$ large enough, $v_j$ will also satisfy \eqref{norm}.
By Lemma \ref{partial-bar}, we have a solution of
$$\bar\partial \sigma_j\,=\,\bar\partial v_j,$$satisfying
$$\int _M\,|\sigma_j|^2_{h_j}\,\omega_j^n\, \leq\, (1+\eta)\,\frac{\bar\epsilon}{(\ell\ell_a)^{n}},$$
where $h_j$ is the hermitian metric on $L^{\ell\ell_a}$ induced by $\omega_j$.
Then $s_j=v_j-\sigma_j$ is a holomorphic section of $L^{\ell\ell_a}$.
From (\ref{elliptic}), we have
\begin{align}\label{C0}
|\sigma_j|_{h_j}(q)\,\leq\, C_1\,\left(\eta\,+\,\frac{\bar\epsilon}{\delta^{2n}}\right )\,\leq\, C_1\,\eta+\delta,\,\,\,\forall q\,\in\, (\phi_j\circ \phi_{a})(V(x;2\delta)).
\end{align}
So $$|s_j|_{h_j}(q)\,\geq\, 1\,-\,C_1\,\eta\,-\,\delta\,\geq\, 1\,-\,2\delta.$$
By the gradient estimate, we have
$$|\nabla s_j|^2_{h_j}\,\leq\, C(n,C_s)\,(\ell\ell_a)^{n-1}\,\int_M\,|s|^2\,\omega_j^n\,\leq\, C(n,C_s)\,(1+\eta)\,\frac{1}{\ell\ell_a}.$$
It follows that if $p$ is sufficiently close to $x$,
$$|s_j|_{h_j}(p)\,\geq\, 1\,-\,2 \delta\,-\,3\, C(n,C_s)\,\delta\,\geq\, \frac{1}{2}.$$
The lemma is proved.
\end{proof}
Now we have to show that for each tangent cone $C_x$, there is a cut-off function satisfying the conditions in Lemma \ref{cut-off}.
For the flat cone $\mathbb{C}_\beta\times \mathbb{C}^{n-1}$ with $\beta \le \bar\beta$,\footnote{$\bar\beta$ is the one given in Lemma \ref{beta1}.}
we have the following cut-off function. Let $\bar\delta \in (0,\frac{1}{3})$
and $\zeta$ be a cut-off function on $\mathbb R$ which satisfies
$0\le\zeta\le 1$,  $|\zeta'|\le 1$ and
\begin{align}\label{eta}
\zeta\,=\,\begin{cases}
0,\,\,\, t\,>\, \log(-\log \bar\delta^3) &\\
1,\,\,\, t\,\le\, \log(-\log \bar\delta). &
\end{cases}
\end{align}
We define a cut-off function $\gamma_{\bar\epsilon}$ on $\mathbb{C}_\beta\times \mathbb{C}^{n-1}$  by

\begin{align}
\gamma_{\bar\epsilon}\,=\,
\begin{cases}  1, \,\,\,{\rm if}\,\, \rho(y)\,=\,{\rm dist} (y, S_x) \,\ge\, \bar\epsilon
 &\\
\zeta\left( \log\left(-\log \left(\frac{\rho(y)}{\bar\epsilon}\right)\right)\right),\,\,\,{\rm if}\,\,\rho(y)\,\le \, \bar\epsilon. &
\end{cases}
\end{align}
Clearly, $\gamma_{\bar\epsilon}\equiv 1$ if $\rho(y)\ge \bar\delta \bar\epsilon$ and $\gamma_{\bar\epsilon}\equiv 0$ if $\rho(y)\le \bar\delta^3  \bar\epsilon$. Moreover, we have
$${\rm supp}\,|\nabla  \gamma_{\bar\epsilon} | \,\subset\subset\,\{y\,|\, \bar\delta^3 \bar\epsilon\,\le\,\rho(y)\,\le\,\bar\delta \bar\epsilon\,\}.$$
Thus we have
\begin{align}\label{cone}\int_{B_{\bar\epsilon^{-1}}(o)} \,|\nabla \gamma_{\bar\epsilon}|^2\,\omega_x^n &\le\,
\left(\int_{ \bar\delta^3 \bar\epsilon\,\le\,\rho(y)\,\le\,
\bar\delta \bar\epsilon}\,\left[\rho\left (-\log \left(\frac{\rho}{\bar\epsilon}\right)\right)\right]^{-2}\, \bar\beta\, \rho d\rho \right)\,\left( \int_{B_{\bar\epsilon^{-1}}(o)'}\,\omega_{\rm flat}^{n-1}\right)\notag\\
&\le\, \frac{2\pi a_{n-1} \bar\beta}{\bar\epsilon^{2(n-1)}} \,\int_{\bar\delta^3}^{\bar\delta}\, \frac{dr}{r(\log r)^2}\notag\\
&=\, \frac{2\pi a_{n-1}\bar\beta}{\bar\epsilon^{2(n-1)}\,(-\log\bar\delta)},
\end{align}
where $B_{\bar\epsilon^{-1}}(o)'$ denotes the $\bar\epsilon^{-1}$-ball in $\mathbb C^{n-1}$ centered at the origin and
$$a_{n-1}\,=\,\frac{1}{(n-1)!}{\rm vol}(B_1(o,\omega_\beta)).$$

We also require the smooth convergence in the support of $\gamma_{\bar\epsilon}$. It will be done in the following lemma.

Now we can prove the local version of the partial $C^0$ estimate.
\begin{lem}\label{limitdivisor}
Assume that for some sequence $\ell_a \to \infty$,
$$C_x\,=\,(C_x,d_x)\,=\,\lim_{a\to \infty}\,(X, \sqrt{\ell_a}\, d,x)\,=\,\mathbb{C}_\beta\times\mathbb{C}^{n-1}.$$
For any $\delta>0$, there exists $r_x>0$ and a Lipschitz map $F_\infty$:
$$F_\infty: B_1(x, {r_x}^{-1} d)\rightarrow B_2(o,\omega_\beta)$$
which is a $\delta$-Gromov-Hausdorff approximation. Moreover,
$F_\infty(B_1(x,r_x^{-1} d)\cap \mathcal{S})$ is a divisor.
\end{lem}
\begin{proof}
Let $p_j\in M$ converge to $x$ in the Gromov-Hausdorff topology. By Lemma \ref{cut-off}, for small $\delta>0$, there is an integer $\ell=\ell_x$ such that for $j$ large enough, we have holomorphic section $s_j$ satisfying
$$|s_j|_{h_j}(p_j)\,\geq\, 1-\delta, \,\,\,\,|\nabla s_j|^2_{h_j}\,\leq\, C(n,C_s)\,(1+\eta)\,\frac{1}{\ell\ell_a}.$$
Putting $r_x=\frac{1}{2 C(n,C_S)\,\ell_x}$, we have
$$|s_j|_{h_j}(q)\,\geq\, \frac{1}{2}\,-\,\delta\,\,\,\forall q\in B_1(p_j,r_x^{-1} \omega_j).$$
Now from smooth sections $\psi_j(\psi_a(z_\alpha))$ ($\alpha=1,\cdots,n$),
we can solve
$$\bar\partial \sigma_j^\alpha\,=\,\bar\partial \psi_j(\psi_a(\gamma_{\bar\epsilon}\, z_\alpha\, e)).$$
Putting $s_j^\alpha=\phi_j(\psi_a(\gamma_{\bar\epsilon}\,z_\alpha\,e))-\sigma_j^\alpha$, we define
$$F_j\,=\,\left(\frac{s_j^1}{s_j},...,\frac{s_j^n}{s_j}\right).$$
With respect to $r_x^{-1}\,\omega_j$, we have
\begin{align}\label{Lip}
|\nabla F_j|_{r_x^{-1}\,\omega_j}\,\leq\, C(n,C_S).
\end{align}
Similar to (\ref{C0}), for any $q\in (\phi_{j}\circ \phi_{a})(V(x;2\delta))$ and $\eta, \bar\epsilon$ sufficiently small, we have
\begin{align}\label{close}
\left |\frac{s_j^1}{s_j}\,-\,z_\alpha\circ(\phi_{j}\circ \phi_{a})^{-1}\right|(q)\,\leq \,C_1(\eta\,+\,\frac{\bar\epsilon}{\delta^{2n}})\,\leq\, 2\delta.
\end{align}
This says that $F_j$ is a $2\delta$-Gromov-Hausdorff approximation over $V(x,2\delta)$. By (\ref{Lip}), we know that $F_j$ is a $6C(n,C_s)\delta$-Gromov-Hausdorff approximation over $B_1(p_j,r_x^{-1}\,\omega_j)$ since $T_{2\delta}(V(x,2\delta))\supseteq B_2(o,\omega_\beta)$.
Now $\bigcup_{i=l+1}^k F_j(D_i\bigcap B_1(p_j,r_x^{-1}\,\omega_j))$ are divisors in $B_2(o,\omega_\beta)$ and
converges to $F_\infty(B_1(x, r_x^{-1}\,d)\bigcap \mathcal{S})$, where $F_\infty$ is the limit of $F_j$. Then the lemma is proved.
\end{proof}
Similarly, for the iterated tangent, we have
\begin{lem}\label{local}
Let $C_x\,=\,\lim_{a\to \infty}\,(X, \sqrt{\ell_a}\, d,x)$ be a tangent cone. For any $y\in \mathcal{S}_{2n-2}(C_x)\setminus \mathcal{S}_{2n-4}(C_x)$ and any $\delta>0$,
there is a small $s_y >0$ and a Lipschitz map $F_\infty$ from $B_1(x,s_y^{-2}\,\omega_x)$ to $B_2(o, \omega_\beta)$ which is a $\delta$-Gromov-Hausdorff approximation.
Moreover, $F_\infty(B_1(y, s_y^{-2}\,\omega_x)\bigcap \mathcal S_x)$ is a divisor.
\end{lem}
\begin{proof}
Since $y\in \mathcal{S}_{2n-2}(C_x)\setminus \mathcal{S}_{2n-4}(C_x)$, the iterated tangent cone $(C_x)_y$ at $y$ is isomorphic to $\mathbb{C}_\beta\times \mathbb{C}^{n-1}$. So we can do the same arguments on $(C_x)_y$ as we did
in the proof of Lemma \ref{limitdivisor} to get the conclusion.
\end{proof}

\begin{prop}\label{partial}
There exists a positive integer $\ell$ and a positive number $c_\ell$ so that the partial $C^0$-estimate holds, i.e. $\rho_\ell(M,\omega_j)\geq c_\ell.$
\end{prop}
\begin{proof}
By choosing a subsequence, we can assume that $(M,\omega_j)$ converges to a metric space $(X,d)$. To prove the Proposition, we only need to
prove that for any $x\in X$ and $p_j\rightarrow x$, there exists a large integer $\ell=\ell_x$ and a positive number $c_x>0$ satisfying:
$$\rho_{\ell}(M,\omega_j)(x)\,\geq\, c_x.$$
It suffices to construct the required cut-off function on every tangent cone.

At first, by Lemma \ref{local}, for any $y\in \mathcal{S}_{2n-2}(C_x)\setminus \mathcal{S}_{2n-4}(C_x)$, we have a ball $B_{6 s_y}(y)$ with respect to $\omega_x$, such that $\mathcal{S}_x\bigcap B_{6 s_y}(y)$ is a local divisor, that is,
there is a homeomorphism $F_y$ from $B_{6s_y}(y)$ to a tangent cone $(C_x)_y$ such that $F_y(\mathcal{S}_x\bigcap B_{6 s_y}(y))$ is a divisor.
Put
$$\hat{\mathcal S}_x\,=\,\mathcal{S}_x\setminus \bigcup_{y\in \mathcal{S}_{2n-2}(C_x)\setminus \mathcal{S}_{2n-4}(C_x)}B_{s_y}(y).$$
Clearly, it is a closed subset of $\mathcal{S}_{2n-4}(C_x)$. Since the Hausdorff dimension of $\mathcal{S}_{2n-4}(C_x)$ is strictly small than $2n-3$,
there is a finite cover of $\hat{\mathcal S}_x\cap B_{{\bar\epsilon}^{-1}}(o,\omega_x)$ by balls $B_{r_a}(q_a)$ such that
$$\Sigma \,r_\alpha^{2n-3}\,\leq\, \bar\epsilon.$$
So we have a cut-off function $\chi$
such that
$$\chi\equiv1 \text{ in }\bigcup B_{r_a}(q_a), \,\,\,\,\int_{C_x}\,|\nabla \chi|^2\,\omega_x^n\,\leq\, \bar\epsilon.$$
We're going to construct a finite cover of $B_{\bar\epsilon^{-1}}(o,\omega_x)\setminus \bigcup B_{r_a}(q_a)$.

Now we cover $\mathcal{S}_x\bigcap (B_{\bar\epsilon^{-1}}(o,\omega_x)\setminus \bigcup B_{r_\alpha}(q_\alpha))$ by finitely many balls
$B_{s_b}(y_b)$. Let $\zeta$ be the cutoff function in (\ref{eta}). For each ball $B_{s_b}(y_b)$, we abbreviate $F_b = F_{y_b}$ and put
$$\gamma_{\bar\epsilon,b}(x')
\,=\,\zeta\left(\log\left(-\log\left( \frac{|f_b|(F_b(x'))}{\hat\epsilon}\right)\right)\right),$$
where $\hat\epsilon = \hat\epsilon(\bar\epsilon)$ and $f_b$ is the defining function of the divisor $F_y(\mathcal{S}_x\bigcap B_{6 s_b}(y_b))$ in $B_{6 s_b}(y_b)$.
For each $B_{6s_b}(y_b)$, we choose a smooth function $\xi_b$ with support in $ B_{1+\bar\epsilon^{-1}}(o,\omega_x)$ satisfying:

\vskip 0.1in
\noindent
i) $0\le \xi_b\le 1$, $|\nabla \xi_b|\le s_b^{-1}$;
\vskip 0.1in
\noindent
ii) ${\rm supp}(\xi_b)\subset  B_{6s_b}(y_b)$;
\vskip 0.1in
\noindent
iii) $\sum_b\,\xi_b\equiv 1$, ~ on $\mathcal{S}_x\bigcap (B_{\bar\epsilon^{-1}}(o)\setminus \bigcup B_{r_a}(q_a))$.
\vskip 0.1in
Then $\{\xi_b\}$ and $1-\sum_b\xi_b$ form a partition of unit for the cover $\{B_{6 s_b}(y_b),\,\bigcup B_{r_a}(q_a)\}$ on $B_{1+\bar\epsilon^{-1}}(o)$. Then we can patch these cut-off functions together to get $$\gamma_{\bar\epsilon}(x')\,=\,(1-\chi)\,(1-\sum\, \xi_b+\sum \xi_b\,\gamma_{\bar\epsilon,b}(x')).$$
As in the proof of Lemma 5.8 in \cite{Ti15}, if $\hat\epsilon$ is sufficiently small, we have
$$\int_{B_{{\bar\epsilon}^{-1}}(o,\omega_x)} |\nabla\gamma_{\bar\epsilon}|^2\,\leq\, \bar\epsilon.$$
Since the metrics converges smoothly in $B_{\bar\epsilon^{-1}}(o, \omega_x)\setminus \bigcup B_{r_\alpha}(q_\alpha)\bigcup B_{s_b}(y_b)$,
we get the cut-off function required.

\end{proof}
\begin{cor}\label{klt}
The Gromov-Hausodrff limit $(X,d)$ of $(M,\omega_j)$ is homeomorphic to a klt log Fano variety with a weak conic K\"{a}hler-Einstein metric.
\end{cor}
\begin{proof}
By Proposition \ref{partial}, we can embed $(M,\omega_j)$ into certain $\mathbb{CP}^N$ by using an orthonormal basis of $H^0(M, L^\ell)$:
$$\iota_j: M\,\rightarrow \,\mathbb{CP}^N,$$
moreover, $\iota_j$ is uniformly Lipschitz.
Denote the Chow limit of $\iota_j(M, D_1,...,D_k)$ by $(M_\infty, D_1^\infty,...,D_k^\infty)$. Then we have a map:
$$\iota_\infty: (X,d)\rightarrow M_\infty.$$ As showed in [DS][Ti15], $\iota_\infty$ is a homeomorphism and $M_\infty$ is a normal variety and $D_i^\infty(1\leq i\leq k)$ are Weil divisors in $M_\infty$. Now we consider the metric $(\iota_\infty)_*\omega_\infty$ on $M_\infty$.
Choose a local coordinate chart $(U,z_\alpha)$ around any smooth point $p\in M_\infty$ such that $O(1)$ is a trivial bundle and $D_i^\infty\bigcap U$ is defined by $f_i=0$. We can also regard $U$ as an open subset in ${\mathbb C}^n$. By the arguments in the proof of Lemma \ref{limitdivisor}, we can find holomorphic homeomorphisms
$$F_j: V_j\subset M\,\rightarrow \,U\,\subset \mathbb{C}^n$$
such that $(F_j)_*\omega_j$ converge to $\omega_\infty$ smoothly outside $\bigcup_{i=1}^kD_i^\infty$. Now we can write $(F_j)_*\omega_j=\sqrt{-1}\,\partial\bar{\partial} \varphi_j$ on $U$ such that for some $C>0,$
$$|\varphi_j|\,\leq\, C,\,\,\,\, \sqrt{-1}\,\partial\bar{\partial}\,\varphi_j\,\geq\,C^{-1}\,\omega_{Euc}.$$
Also, $\varphi_j$ satisfies
$$(\sqrt{-1}\,\partial\bar{\partial}\,\varphi_j)^n\,=
\,e^{-t_j\varphi_j}\Pi_{i=1}^k|f^j_i|^{2(\beta_j^i-1)}\,|\sigma_j|^2,$$ where $\sigma_j$ is a non-vanishing holomorphic section in $\Gamma(U, (F_j)_{*}K_{M})$ and $f^j_i$  is a defining function of $F_j(D_i)$.
Since $\beta_j^i\rightarrow 1 (1\leq i\leq l)$, $\varphi_j$ converges to bounded function $\varphi_\infty$ on $U$ satisfying
\begin{align}\label{wke}
(\sqrt{-1}\,\partial\bar{\partial}\,\varphi_\infty)^n\,
=\,e^{-t_\infty\phi_\infty}\Pi_{i=l+1}^k|f_i|^{2(\beta_\infty^i-1)}
|\sigma_\infty|^2,
\end{align}
where $\sigma_\infty$ is a non-vanishing holomorphic section in $\Gamma(U, K_{M_\infty})$.
So $\omega_\infty$ is a weakly conic K\"{a}hler-Einstein metric. For any point $z\in M_\infty$, we can find a Zariski open subset $V$ so that (\ref{wke}) holds in the regular part of $V$.
It is easy to see that
$$\Pi_{i=l+1}^k|f_i|^{2(\beta_\infty^i-1)}|\sigma_\infty|^2$$
is integrable over $V$. Thus, it follows that $(M_\infty, \sum_{i=l+1}^k(1-\beta^i_\infty)D_i^\infty)$ is a klt log pair.
\end{proof}
\section{Existence of conic K\"{a}hler-Einstein metrics}
In this section, we will prove our main theorem. First, we recall the definition of log $K-$polystable.
\begin{defi}
A normal test configuration $(\mathcal{X},\mathcal{D},\mathcal{L})$ for $(M,D,L)$, where $D=\Sigma_{i=1}^k\,(1-\beta^i)D_i$, consists of:

\vskip 0.1in
\noindent
i) a normal log-pair $(\mathcal{X},\mathcal{D})$ with a $C^*$-equivariant morphism $\pi: (\mathcal{X},\mathcal{D})\rightarrow \mathbb{C}$,

\vskip 0.1in
\noindent
ii)$\mathcal{L}$ is an equivariant $\pi$-ample $\mathbb{Q}$-line bundle such that
$$(\mathcal{X}_t,\mathcal{D}_t,\mathcal{L}_t)\cong (M,D, L),\,\,\,\forall t\neq 0.$$
\vskip 0.1in
\noindent
Let $(\bar{\mathcal{X}},\bar{\mathcal{D}},\bar{\mathcal{L}})$ be the natural equivariant compactification of $(\mathcal{X},\mathcal{D},\mathcal{L})$. The log-CM weight is defined by
$$w_{lcm}(\mathcal{X},\mathcal{D},\mathcal{L})\,=\,\frac{n\, \bar{\mathcal{L}}^{n+1}\,+\,(n+1)\,\bar{\mathcal{L}}^n\,(K_{\bar{\mathcal{X}}/\mathbb{CP}^1}\,+\,\bar{\mathcal{D}})}{(n+1)\,L^n}.$$

We say $(M,D,L)$ log K-semistable if
$w_{lcm}(\mathcal{X},\mathcal{D},\mathcal{L})\geq 0$ for any normal test configuration $(\mathcal{X},\mathcal{D},\mathcal{L})$.

We say $(M,D,L)$ log K-polystable if $w_{lcm}(\mathcal{X},\mathcal{D},\mathcal{L})\geq 0$ for any normal test configuration $(\mathcal{X},\mathcal{D},\mathcal{L})$,
and the equality holds if and only if $(\mathcal{X},\mathcal{D},\mathcal{L})\cong (X,D,L)\times \mathbb{C}.$
\end{defi}
If the central fiber is a klt log Fano variety $(M_0,\Sigma_{i=1}^k(1-\beta^i)D^0_i, L_0)$, we have an embedding of $M_0$ into $\mathbb{CP}^N$ using $H^0(M_0,L^\ell)$ for a large $\ell$. Let $\bar\omega_0$ be the restriction of the Fubini-Study metric:
 $$ \bar\omega_0\,=\,\frac{1}{\ell}\,\iota^*\omega_{FS},\,\,\,{\rm where}\,\,\iota: M_0\rightarrow \mathbb{CP}^N.$$
 Then as in \cite{DT}, we have
$$w_{lcm}(\mathcal{X},\mathcal{D},\mathcal{L})\,=\,-\,\int_{M_0} \,\theta_X({\rm Ric}(\bar\omega_0)\,-\,\bar\omega_0)\wedge\bar\omega_0^{n-1}\,+\,\Sigma_{i=1}^k\,(1-\beta^i)\,\int_{D^0_i}\,\theta_X\,\bar\omega_0^{n-1},$$
where $X$ generates the $\mathbb{C}^*$-action on $M_0$ and $\theta_X$ is the potential of $X$. As before, we will
denote $\Sigma_{i=1}^k\,(1-\beta^i)\,D^0_i$ by $(1-\beta)\,D^0$ and the righthand side above by
$Fut_{(M_0, (1-\beta)D^0)}(X).$

Now we prove our main theorem.
\begin{theo}
If $(M,(1-\beta)D)$ is log K-polystable, then there exists a conic K\"{a}hler-Einstein metric.
\end{theo}

\begin{proof}

Choose a large integer $\lambda$ large enough such that there is a smooth divisor $E\in|\lambda L|$. Fix a hermitian metric $h_E$ on $\lambda L$ such that
$$R(h_E)\,=\,\lambda\,\omega_0.$$
Consider the following continuity method through conic K\"ahler-Einstein metrics:
$${\rm Ric}(\omega_t)\,=\,t\omega_t\,+\,\frac{1-t}{\lambda}\,[E]\,+\,(1-\beta)\,[D].$$
It is equivalent to the following complex Monge-Amp\'{e}re equation:
\begin{align}\label{continuity}
(\omega_0\,+\,\sqrt{-1}\,\partial\bar{\partial}\,\phi_t)^n\,=\,e^{-t\phi_t}\,|s_E|^{-2\frac{1-t}{\lambda}}_{h_E}\,|s_D|^{-2(1-\beta)}_{h_D}\,\Omega,
\end{align}
where $\Omega$ is the volume form given by
\begin{equation}\label{tian5.1}
{\rm Ric}(\Omega)\,=\,\omega_0\,+\,(1-\beta)\,R(h_D),\,\,\,\,\int_M\,\Omega \,=\,\int_M\,\omega_0^n\,=\,V.
\end{equation}
The corresponding Lagrangian, also often referred as Ding functional, is
$$F_t\,=\,J(\phi)\,-\,\int_M\,\phi\,\omega_0^n\,-\,\frac{1}{t}\,\log \left(\frac{1}{V}\,\int_M \,e^{-t\phi}\,|s_E|^{-2\frac{1-t}{\lambda}}_{h_E}\,|s_D|^{-2(1-\beta)}_{h_D}\,\Omega\right).$$
\begin{lem}\label{small}
There exists a positive number $t_0=t_0(M,\lambda)$ such that for $t\in (0,t_0)$, $F_t$ is coercive.
\end{lem}
\begin{proof}
For $\phi\in PSH(M,\omega_0)$ with $sup \,\phi=0$, we have
$$\int_M\,e^{-\delta\phi}\,\Omega\,\leq\, A$$
for some $\delta\leq \delta_0=\delta_0(M,\omega_0)$. We may assume that $\lambda\geq 2$, then there is $p>1$ such that
$$\left (\int_M \,\left [|s_E|^{-2\frac{1-t}{\lambda}}_{h_E}\,|s_D|^{-2(1-\beta)}_{h_D}\right]^p\,\Omega\right )^{\frac{1}{p}}\,\leq\, C_p.$$
Taking $t_0=\frac{p-1}{p}\,\delta_0$ and using H\"{o}lder's inequality, we have
$$\int_M \,e^{-t\phi},|s_E|^{-2\frac{1-t}{\lambda}}_{h_E}\,|s_D|^{-2(1-\beta)}_{h_D}\,\Omega\,\leq\, A^{\frac{p-1}{p}}\,C_p.$$
It follows that for $t\in (0,t_0)$,
$$F_t\,\geq\, J\,-\,C.$$
\end{proof}
From this, we know that $Aut^0(M,E+D)=1$. Now put
$$I\,=\,\{\,t\in[0,1]\,|\,(\ref{continuity})\text{ is  solvable}\,\}.$$ By Lemma \ref{small} and Theorem A in \cite{GP}, we know that $[0,t_0)\subset I$.
Since there is no non-trivial holomorphic vector field tangential to $E$ and $D$, $I$ is open.
We are going to show that $I$ is closed using log K-polystability.

Assuming that $(M,\omega_{t_i}, \frac{1-t_i}{\lambda}E +(1-\beta)D)$ converges to a metric space $(X,d)$ in the Gromov-Huasdroff topology as $t_i\in I \rightarrow T$. From Corollary \ref{klt}, for some large integer $\ell$, we can embed $M$ into $\mathbb{CP}^n$ using orthonormal bases of $H^0(M,L^\ell)$ with respect to $\omega_{t_i}$, and we can also show that $(X,d)$ is homeomorphic to the Chow limit $(M_\infty,\omega_\infty, \frac{1-T}{\lambda}E_\infty+(1-\beta)D_\infty)$ which is a klt log Fano variety with weak conic K\"{a}hler-Einstein metric: .
\begin{lem}\label{reductive}
$Aut^0(M_\infty, \frac{1-T}{\lambda}E_\infty+(1-\beta)D_\infty)$ is reductive.
\end{lem}
If $T=1$, $Aut^0(M_\infty, \frac{1-T}{\lambda}E_\infty+(1-\beta)D_\infty)$ is understood as
$Aut^0(M_\infty, (1-\beta)D_\infty)$.
\begin{proof}
We will use the method for proving Lemma 6.9 in \cite{Ti15}. We will prove that any holomorphic
field in the Lie algebra $\eta_\infty$ of $Aut^0(M_\infty, \frac{1-T}{\lambda}E_\infty+(1-\beta)D_\infty)$ is a complexification of a Killing field on $(M_\infty,\omega_\infty).$

We only consider $T<1$, since the proof for $T=1$ is identical.
Now $M_\infty$ is embedded in $\mathbb{CP}^N$ by $H^0(M_\infty, L^\ell)$. If $X\in \eta_\infty$ is a holomorphic vector field in $\mathbb{CP}^N$, then there is a bounded function $\theta_\infty$ satisfying
$$i_X\omega_\infty\,=\,\sqrt{-1}\,\bar\partial \,\theta_\infty,\,\,\,\, \Delta \,\theta_\infty\,=\, -\,T\,\theta_\infty,\,\,
\text{ in } M_\infty\setminus \mathcal{S}.$$
This is exactly the same as Lemma 6.9 in \cite{Ti15}.
Putting $\theta_\infty=u+\sqrt{-1}v$, we are going to show that $\nabla u$ corresponds to a Killing vector field. At first,
we show that $u$ is the limit of eigenfunctions $u_j$ on $(M,\omega_j)$ such that $\Delta \,u_j\,=\,-\,\lambda_j u_j$ with
$\lambda_j\rightarrow T$.
Denote the set of such limit eigenfunctions by $\tilde\Lambda_T$ which is a subset of $\Lambda_T$ consisting of all bounded eigenfunctions with eigenvalue $T$. If $\tilde\Lambda_T\neq \Lambda_T$, we can find $u\in \Lambda_T $ such that
$$\int_{M_\infty}\,u^2\,\omega_\infty^n\,=\,1,\,\,\,\, \int_{M_\infty}\,u u_a\,\omega_\infty^n\,=\,0,$$
where $\{u_a\}_{1\leq a\leq k}$ is an orthonormal basis of $\tilde\Lambda_T$.
Because $\mathcal{S}$ is a subvariety which is contained in a divisor, as in Lemma \ref{moser}, we have a cut-off function in $M_\infty$ satisfying
$$\int_{M_\infty}\,|\nabla \gamma_\epsilon|^2\,\omega_\infty^n\,\leq\, \epsilon.$$
On the support $K_\epsilon$ of $\gamma_\epsilon$, by Lemma \ref{angle}, we have $\phi_j: K_\epsilon \rightarrow (M,\omega_j)$ such that $\phi_j^*\omega_j\rightarrow\omega_\infty$ smoothly. So by taking a subsequence if necessary, we can take $\epsilon_j\rightarrow 0$ such that
$u_j\,=\,(\gamma_{\epsilon_j}\,u)\circ \phi^{-1}_{j}$ satisfy
$$\lim_{j\rightarrow\infty}\,\int_M\, |\nabla u_j|^2\,\omega_{j}^n\,=\,T, \,\,\,\,\lim_{j\rightarrow\infty}\,\int_M\, u_j^2\,\omega_{j}^n\,=\,1.$$
For each $a$, there are eigenfunctions $u_{a,j}$ of $(M,\omega_j)$ which converge to $u_a$. So if $j$ large enough,
$u_j,u_{1,j},...,u_{k,j}$ is a $k+1$ dimensional subspace. We can find an eigenfunction $u_{0,j}$ orthogonal to $u_{a,j}(1\leq a\leq k)$ with eigenvalue not bigger than $T+\nu_j$ with $\nu_j\rightarrow 0$.
By Lemma \ref{eigenvalue}, we know that the eigenvalue is not less than $t_{j}$. So $u_{0,j}$ will converge to an element in $\tilde\Lambda_T$. It leads to
a contradiction. Now we know that $u$ is the limit of eigenfunctions $u_j$. By (\ref{bochner}), we have
$$\int_M \,|\bar\nabla\bar\partial u_j|^2\,\omega_{j}^n\, \leq\, (\lambda_j-t_{j})\,\int_M \,|\nabla u_j|^2\,\omega_{j}^n\,\rightarrow\, 0.$$
It follows that $\bar\nabla\bar\partial u\,=\,0$ which means that $Y=\nabla u$ is a Killing vector field. As showed in Lemma 6.9 in \cite{Ti15}, $Y$ can be extended to a holomorphic vector field on $\mathbb{CP}^n$. So we have proved that $\eta_\infty$ is the complexification of the Killing vector fields.
\end{proof}

There is another approach to establishing the reductivity by using Theorem 5.1 in \cite{BBEGZ}.\footnote{In its new version posted in 2016,
an appendix was added to fill in necessary arguments for a complete proof.}

By Luna's slice lemma (cf. \cite{D}), there is a test-configuration of $(M, (1-\beta)D)$ with $(M_\infty, (1-\beta)D_\infty)$ as the central fiber. On $(M_\infty,\omega_\infty)$, the following Futaki invariant vanishes:
$$Fut_{M_\infty,(1-\beta)D_\infty}(X)\,-\,(1-T)\,\int_M\,\theta_X\,\omega^n\,+\,\frac{1-T}{\lambda}\,\int_E\,\theta_X\,\omega^{n-1}\,=\,0.$$
We will adapt the argument in \cite{DT} to prove this. Using the embedding of $H^0(M,L^\ell)$, we know that
$$\omega_j\,=\,\frac{1}{\ell}\,\omega_{FS}\,-\,\frac{1}{\ell}\,\sqrt{-1}\,\partial\bar\partial\,\log \rho_\ell(M,\omega_j).$$
By Proposition \ref{partial} and gradient estimate, we know that $\rho_\ell(M,\omega_j)$ is uniformly Lipschitz.
Denoting the limit of $\rho_\ell(M,\omega_j)$ by $\rho_\ell(M_\infty,\omega_\infty)$, we can also choose a bounded hermitian metric $h_\infty$ on the limit $\mathbb{Q}$-line $L_\infty=\frac{1}{\ell}\textsl{O}(1)$ with $R(h_\infty)=\omega_\infty$ in the regular part.
Then we have
\begin{align}
\rho_\ell(M_\infty,\omega_\infty)\,=\,\sum_{i=0}^N\, |s_i|_{h_\infty^{\otimes l}}^2 \,\,\text{ for  an orthonormal basis } \{s_i\}. \notag
\end{align}
For simplicity, denoting
$$\bar\omega_0\,=\,\frac{1}{\ell}\,\omega_{FS}, \,\,\,\,\varphi_\infty\,=\,-\,\frac{1}{\ell}\,\log \rho_\ell(M_\infty,\omega_\infty),$$
we define
$$\omega_s\,=\,\bar\omega_0\,+\,s\,\sqrt{-1}\,\partial \bar\partial\, \varphi_\infty\,=\,s\,\omega_\infty\,+\,(1-s)\,\bar\omega_0.$$
The log Ricci potential $f_0$ of $\bar\omega_0$ is given by
$$\sqrt{-1}\,\partial\bar\partial \,f_0\,=\,{\rm Ric}(\bar\omega_0)\,-\,T\,\bar\omega_0\,-\,\frac{1-T}{\lambda}\,E_\infty\,-\,(1-\beta)\,D_\infty.$$
We know that
$$Fut_{M_\infty,(1-\beta)D_\infty}(X)\,-\,(1-T)\,\int_M\,\theta_X\,\omega^n\,+\,
\frac{1-T}{\lambda}\,\int_E\,\theta_X\,\omega^{n-1}\,=\,\int_{M_\infty}\,X(f_0)\,\omega_0^n.$$
Put
$$f_s\,=\,-\,\log\left(\frac{\omega_s^n}{\bar\omega_0^n}\right)\,-\,T s\,\varphi_\infty\,+\,f_0$$
which is the log Ricci potential of $\omega_s$. Denoting by $\upsilon$ the potential of $X$ with respect to $\bar\omega_0$:
$$i_X\bar\omega_0\,=\,\sqrt{-1}\,\bar\partial\, \upsilon, $$
then we have
$$i_X\omega_s\,=\,\sqrt{-1}\,\bar\partial\,\theta_s,\,\, \text{ for }
\theta_s\,=\,(1-s)\,\upsilon\,+\,s\,\theta_\infty\,=\,\upsilon\,+\,s\,X(\varphi_\infty).$$
From Lemma \ref{reductive}, we know that $|X|_{\omega_\infty}=|\nabla \theta_\infty|$ is bounded.
Since
$$\omega_0\,\leq\, C \omega_\infty,\,\,\, s\,\omega_\infty\,\leq\,\omega_s\,\leq \,C\,\omega_\infty,$$
we know that $|X|_{\omega_s}\leq C$ and $|\Delta_\infty\,\upsilon|\leq C$. Consequently, we have
$|\Delta_\infty\theta_s|\leq C$ and $|\Delta_s\theta_s|\leq \frac{C}{s}.$
Now $$X\left(\log\left(\frac{\omega_s^n}{\bar\omega_0^n}\right)\right)\,=\,div_s(X)\,-\,div_0(X)\,=\,\Delta_s \theta_s\,-\,\Delta_0 \upsilon$$ is bounded,
so we have
\begin{align}
\int_{M_\infty}\,X(f_s)\,\omega_s^n&=\,
\int_{M_\infty}\,(\Delta_0\,\upsilon\,-\,\Delta_s\,\theta_s)\,\omega_s^n\,-\,Ts\,\int_{M_\infty}\,X(\varphi_\infty)\,\omega_s^n\,+\,
\int_{M_\infty}\,X(f_0)\,\omega_s^n\notag \\
&=\,\int_{M_\infty}\,\Delta_0\,\upsilon\,\omega_s^n\,-\,Ts\,\int_{M_\infty}\,X(\varphi_\infty)\,\omega_s^n\,+\,
\int_{M_\infty}\,X(f_0)\,\omega_s^n \notag.
\end{align}
Taking $s\rightarrow 0$, we have
$$\lim_{s\rightarrow 0}\,\int_{M_\infty}\,X(f_s)\,\omega_s^n\,=\,\int_{M_\infty}\,X(f_0)\,\omega_0^n.$$

We are going to show that the integral
$$\int_{M_\infty}\,X(f_s)\,\omega_s^n$$
is independent of $s$. From $\frac{d}{ds}f_s \,=\,- \Delta_s \varphi_\infty\,-\,T\varphi_\infty$, we have
\begin{align}\label{invariant}
\frac{d}{ds}\int_{M_\infty}\,X(f_s)\,\omega_s^n&=\,\int_{M_\infty}\,[X(-\Delta_s\,\varphi_\infty\,-\,T\varphi_\infty)
\,+\,X(f_s)\,\Delta_s\,\varphi_\infty]\,\omega_s^n\notag\\
&=\,\int_{M_\infty}\,[(\Delta_s\,\theta_s\,+\,X(f_s))\,\Delta_s\,\varphi_\infty\,-\,T\,X(\varphi_\infty)]\,\omega_s^n\notag\\
&=\int_{M_\infty}\,[\Delta_s\,\theta_s\,+\,\theta_s\,+\,T\,X(f_s)]\,\Delta_s\,\varphi_\infty\,\omega_s^n.
\end{align}
The integration by parts is guaranteed by the following two lemmas.
\begin{lem}
$|\nabla \Delta_s \,\varphi_\infty|_{\omega_s}$ is bounded.
\end{lem}
\begin{proof}
Putting $h_s=h_\infty e^{(1-s)\varphi_\infty}$,
we have
$$-\,s\varphi_\infty\,=\,\frac{1}{\ell}\,\log (\sum_{i=0}^N\, |s_i|^2_{h_s}),$$
we also have
$$\sqrt{-1}\,\partial\bar\partial\, \log (\sum_{i=0}^N \,|s_i|^2_{h_s})\,=\,\sqrt{-1}\,\frac{\sum_{i=0}^N\,\langle \nabla s_i,\nabla s_i\rangle}{\sum_{i=0}^N \,|s_i|^2_{h_s}}\,-\,\sqrt{-1}\,\frac{\sum_{i=0}^N\,\langle \nabla s_i,s_i\rangle\sum_{i=0}^N\,\langle s_i,\nabla s_i\rangle}{(\sum_{i=0}^N\, |s_i|^2_{h_s})^2}\,-\,\omega_s.$$
From these, we can deduce that $|\nabla \Delta_s \,\varphi_\infty|_{\omega_s}$ can be bounded by the gradient of $s_i$. The lemma is proved.
\end{proof}
\begin{lem}
If $|\nabla u|_{\omega_s}$, $|\nabla v|_{\omega_s}$ and $|\Delta_s \,u|$ are all bounded, then
$$\int_{M_\infty}\,\langle \nabla u,\nabla v\,\rangle\,\omega_s^n\,=\,-\,\int_{M_\infty}\,v\, \Delta u\, \omega_s^n.$$
\end{lem}
\begin{proof}
Let $\gamma_\epsilon$ be the cut-off function on $M_\infty$ as in Lemma \ref{reductive}. We have
$$0=\int_{M_\infty}\,div(\gamma_\epsilon v \,\nabla u)\,\omega_s^n\,=\,\int_{M_\infty}\,\gamma_\epsilon\,\langle \nabla u,\nabla v\rangle\,\omega_s^n+
\int_{M_\infty}v\,\langle \nabla\gamma_\epsilon ,\nabla u\rangle\,\omega_s^n+
\int_{M_\infty}\,\gamma_\epsilon v\,\Delta_s u\,\omega_s^n.$$
Since $$\int_{M_\infty}\,v\,\langle \nabla\gamma_\epsilon ,\nabla u\rangle\,\omega_s^n\,\leq\,
C\,\int_{M_\infty}\,|\nabla \gamma_\epsilon|^2\,\omega_s^n,$$
the lemma is proved by taking $\epsilon\rightarrow 0.$
\end{proof}
A direct computation shows that $\bar\partial(\Delta_s\theta_s+\theta_s+T X(f_s))=0$, so
$\Delta_s\theta_s+\theta_s+TX(f_s)$ is a bounded holomorphic function which must be constant. From (\ref{invariant}),
we know that
$$\frac{d}{ds}\left(\int_{M_\infty}\,X(f_s)\,\omega_s^n\right)\,=\,0.$$
Since $\omega_\infty$ is a weakly conic K\"{a}hler Einstein metric, $f_1=0$ and we have $$Fut_{M_\infty,(1-\beta)D_\infty}(X)-(1-T)\int_M\,\theta_X\,\omega^n+\frac{1-T}{\lambda}\,\int_E\,\theta_X\,\omega^{n-1}\,=\,0.$$
The vanishing of log Futaki invariant can be also obtained from the log $K$-polystability of singular varieties admitting weakly conic K\"{a}hler Einstein metric (Theorem 4.8 in \cite{Ber}).

Denote $$f(t)\,=\,Fut_{X_\infty,(1-\beta)D_\infty}-(1-t)\,\int_M\,\theta_X\,\omega^n+\frac{1-t}{\lambda}\,\int_E\,\theta_X\,\omega^{n-1},$$
then from Lemma \ref{small}, we know that
 $f(t)>0$ for $t$ small. Since $f(T)$=0, we must have $f(1)\leq 0.$ Because $(M,(1-\beta)D)$ is log K-polystable, $f(1)$ must be equal to $0$ and $X_\infty$ is bi-holomorphic to $M$. It follows that $T\in I$ which means that $I$ is closed. So $I=[0,1]$ and we get the existence of conic K\"{a}hler-Einstein metrics.
\end{proof}


\begin{thebibliography}{99}
\bibitem[An]{An}
Anderson, M. T.: Ricci curvature bounds and Einstein metrics on compact manifolds. J. Amer.
Math. Soc., 2 (1989), no. 3, 455¨C490.
\bibitem[Ber]{Ber}
Berman, R.: K-polystability of $\mathbb{Q}$-Fano varieties admitting K\"{a}hlerEinstein
metrics. Invent. Math., (203) 2016, Issue 3, 973¨C1025.
\bibitem [BBEGZ]{BBEGZ} Berman, R.; Boucksom, S.; Eyssidieux, P.; Guedj V. and Zeriahi A.:
K\"{a}hler-Einstein metrics and the K\"{a}hler-Ricci flow on log Fano varieties.
arXiv:1111.7158
\bibitem[BBGZ]{BBGZ}
Berman, R.; Boucksom, S.;  Guedj, V. and Zeriahi, A.: A variantioanl approach to complex Monge-Amp\`{e}re equations. Pub. Math. de l'IH\'{E}S, 117(2013), Issue 1, 179-245.
\bibitem[Bo]{Bo}
Berndtsson, B.: Brunn-Minkowski type inequality for Fano manifolds
and the Bando-Mabuchi uniqueness theorem. Invent. Math., (200) 2015, Issue 1, 149¨C200.
\bibitem[Bl]{Bl} Blocki Z.: On the regularity of the complex Monge-Amp\'{e}re operator, Complex geometric analysis in Pohang (1997) (Providence, RI), Contemp. Math., vol. 222, Amer. Math. Soc., 1999, pp. 181¨C189.
\bibitem[BEGZ]{BEGZ}
Boucksom, S., Eyssidieux, P., Guedj, V. and Zeriahi, A.: Monge-Amp\`{e}re equations in big cohomology classes. Acta Math., 205 (2010), Issue 2, 199-262.
\bibitem[CC1]{CC1}
Cheeger, J. and Colding, T. H.: On the structure of spaces with Ricci curvature bounded below. I. J. Diff. Geom., 46 (1997), no. 3, 406¨C480.
\bibitem[CC2]{CC2}
Cheeger, J. and Colding, T.H.: On the structure of spaces with Ricci curvature bounded below. II. J. Diff. Geom., 54 (2000), no. 1, 13-35.
\bibitem[CCT]{CCT}
Cheeger, J.; Colding, T. H. and Tian, G.: On the singularities of spaces with bounded Ricci curvature. Geom. Funct. Anal., 12 (2002), no. 5, 873-914.
\bibitem[CDS15]{CDS}
Chen, X.X.; Donaldson S. and Sun S.: K\"{a}hler-Einstein metrics on Fano manifolds III: Limits as cone angle approaches $2\pi$ and completion of the main proof. J. Amer. Math. Soc., 28 (2013), 235-278.
\bibitem[Co]{Co}
Colding, T. H.: Ricci curvature and volume convergence. Ann. of Math., (2) 145 (1997), no. 3,
477-501.
\bibitem[DR]{DR}
Darvas, T. and Rubinstein Y.: Tian's properness conjectures and Finsler geometry of the space of Kahler metrics.  J. Amer. Math. Soc., 30 (2017), 347-387.
\bibitem[DT]{DT}
Ding, W.Y. and Tian, G.:  K\"{a}hler-Einstein metrics and the generalized
Futaki invariant. Invent. Math., 110 (1992), no. 2, 315-335.
\bibitem[D]{D}
Donaldson, S.: Stability, birational transformations and the K\'{a}hler-Einstein problem. Surveys in Differential Geometry, 17. International Press, Boston, 2012.
\bibitem[DS]{DS}
Donaldson, S. and Sun, S.: Gromov-Hausdorff limits of K\"{a}hler manifolds and algebraic
geometry. Acta Math., 213 (2014), no. 1, 63-106.


\bibitem[GP]{GP}
Guenancia, H. and P\v{a}un, M.: Conic singularities metrics with prescribed Ricci curvature: the case of general cone angles along normal crossing divisors. J. Diff. Geom., 103 (2016), no.1, 15-57.
\bibitem[Li11]{Li11}
Li, C.: Remarks on logarithmic K-stability. Commun. Contemp. Math., 17 (2014), no. 2, 1450020, 1-17.
\bibitem[S]{S}
Shen, L.M.: Smooth approximation of conic K\"{a}hler metric with lower Ricci curvature bound. Paci. J. Math., 284 (2016), no. 2, 455¨C474.
\bibitem[Ti97]{Ti97}
Tian, G: K\"{a}hler-Einstein metrics with positive scalar curvature. Invent. Math., 130(1997), no. 1, 1-37.
\bibitem[Ti15]{Ti15}
Tian, G.: K-stability and K\"{a}hler-Einstein metrics. Comm. Pure Appl. Math., 68 (2015), no. 7, 1085-1156.
\bibitem[TW16]{TW16}
Tian, G and Wang, B.: On the structure of almost Einstein manifolds. J. Amer. Math. Soc. 28 (2015), no. 4, 1169-1209.

\end{thebibliography}
\end{document}